\documentclass[12pt,twoside]{amsart}

\usepackage{r-English}
\usepackage[
    letterpaper,
    left=1in,
    right=1in,
    top=1.1in,
    bottom=1.1in,
    includehead,
    headheight=14pt,
    headsep=18pt
]{geometry}
\usepackage{amsmath}
\usepackage{amssymb}

\theoremstyle{plain}
\newtheorem{introthm}{Theorem}

\title[Stationary states for an inner action]
{Stationary states on a $C^*$-algebra for an inner action}
 
\author{Gal Ben Ayun}
\date{}
\begin{document}
\title{Stationary states on a $C^*$-algebra for an inner action}
\date{}
\maketitle

\begin{abstract}
We study stationary states for actions of countable discrete groups on unital separable
$C^*$-algebras. We prove that the stationary state space associated with an inner action
of a subgroup of the unitary group that generates the algebra is a Choquet simplex. We
also give a locality criterion covering Bernoulli shifts. The simplex structure yields a
canonical decomposition of stationary states into tracial and purely nontracial parts.
For inner actions, we characterize extreme stationary states by factoriality of their GNS
von Neumann algebras. We further show that a stationary state is tracial if and only if
its GNS von Neumann algebra is finite, and that it is purely nontracial if and only if this
algebra is of type~$\mathrm{III}$. As applications, for $2\leq d\leq\infty$ the stationary state simplex of
$C^*(\mathbb{F}_d)$ has a Poulsen face, while for a
nontrivial property~$(T)$ group the stationary state simplex of $C^*(\Gamma)$ is not
Poulsen. Finally, we give a sufficient spectral-gap condition for $S_{\mu}(A)$ to be a Bauer simplex and construct a family $(A_d,\Gamma_d,\mu_d)$ satisfying this condition.
\end{abstract}
\section{Introduction}
It is a classical fact that the tracial state space $T(A)$ of a unital $C^*$-algebra
$A$ is a possibly empty Choquet simplex \cite{blackadar2024tracial}. Moreover, a trace
is extreme precisely when its GNS von Neumann algebra is a factor. Thus traces admit a
unique ergodic decomposition whose pieces can be recognized from their GNS
representations.
Let a countable discrete group $\Gamma$ act on $A$, and let $\mu$ be a generating
probability measure on $\Gamma$. Recall that a state $\varphi$ is $\mu$-stationary if
\[
    \varphi=\sum_{g\in\Gamma}\mu(g)(g\mathbin{\cdot}\varphi).
\]
Stationary states are the analogues of stationary measures for the random walk driven
by $\mu$. They form a nonempty weak$^*$ compact convex set, denoted by $S_\mu(A)$,
even when invariant states do not exist. Stationary $C^*$-dynamical systems were
studied systematically by Hartman and Kalantar \cite{hartman2023stationary}.
Suppose now that \mbox{$\Gamma\leq U(A)$} acts by conjugation and generates $A$ as a
$C^*$-algebra. The $\Gamma$-invariant states are then exactly the tracial states, so
\[
    T(A)=S^\Gamma(A)\subset S_\mu(A)\subset S(A).
\]
It is therefore natural to ask whether the simplex structure of $T(A)$
survives on the larger space $S_\mu(A)$.
For $A=C^*(\Gamma)$, the same question can be phrased entirely in terms of functions
on $\Gamma$. States correspond to normalized positive-definite functions, traces to
the conjugation-invariant ones, and stationary states to those satisfying
\[
    \varphi(h)
    =\sum_{g\in\Gamma}\mu(g)\varphi(g^{-1}hg),
    \qquad h\in\Gamma.
\]
In other words, $S_\mu(C^*(\Gamma))$ consists of the normalized positive-definite
functions which are harmonic for the conjugation random walk. The question was posed to the author by Hartman and Vigdorovich, motivated by forthcoming work on character stiffness for certain classes of amenable groups, in which an affirmative answer plays a crucial role \cite{hartman-vigdorovich-prep}.
The answer to the
question above is affirmative.

\begin{introthm}[Theorem~\ref{thm:stationaries.form.simplex}]
Let $A$ be a unital separable $C^*$-algebra, let \mbox{$\Gamma\leq U(A)$} be a countable
subgroup that generates $A$ as a $C^*$-algebra and acts on $A$ by conjugation, and let
\mbox{$\mu\in\prob(\Gamma)$} be a generating measure. Then $S_\mu(A)$ is a Choquet
simplex.
\end{introthm}
 We give two proofs of the theorem: one based on Batty's
characterization of invariant-state simplices \cite[Theorem~6.1]{cjk1982invariant},
and one based on the GNS characterization of extreme stationary states developed
below.
Batty's criterion also supplies another family of actions for which the stationaries form a simplex. Suppose that $\Gamma$ acts on a countable set $I$ and that $A$ contains a
norm-dense local $*$-subalgebra
\mbox{$\displaystyle A_{\mathrm{loc}}=\bigcup_{F\subset I\text{ finite}}A_F$}.
If \mbox{$\alpha_g(A_F)=A_{gF}$}, disjointly supported algebras commute, and every
$\Gamma$-orbit in $I$ is infinite, then Theorem~\ref{thm:Choquet.simplex.finite.sets.}
shows that $S_\mu(A)$ is again a Choquet simplex. In particular, the result applies to
Bernoulli shifts over infinite countable groups.
This may be viewed as a stationary counterpart of the classical theory of invariant
states on asymptotically abelian $C^*$-dynamical systems
\cite{doplicher1969invariant}.

The GNS representation gives an operator-algebraic description of the extreme
boundary of $S_\mu(A)$. For \mbox{$\varphi\in S_\mu(A)$}, set
\mbox{$M_\varphi=\pi_\varphi(A)''$}, and denote by $\omega_\varphi$ the normal extension of
$\varphi$ and by $T_\varphi$ the induced Markov operator on $M_\varphi$. We prove the
following characterization.

\begin{introthm}[Theorem~\ref{thm:extreme-factor}]
Let $\varphi\in S_{\mu}(A)$. The following are
equivalent:
\begin{enumerate}
    \item $\varphi$ is an extreme point of $S_\mu(A)$;
    \item $M_\varphi$ is a factor;
    \item \mbox{$\mathrm{Fix}(T_\varphi)=\mathbb{C}1$};
    \item $\omega_\varphi$ is the unique normal $\mu$-stationary state on $M_\varphi$.
\end{enumerate}
\end{introthm}
The factorial characterization also has a partial converse at the level of convex structure.
For a semigroup of unital completely positive maps on $A$, we define the compact convex set
$$S_{\Sigma}(A)=\{\varphi\in S_{\mu}(A):\varphi\circ \sigma=\varphi,\,\forall\sigma\in\Sigma\}.$$
We prove that if the extreme points of
$S_\Sigma(A)$ are precisely the states with factorial GNS algebra, then
$S_\Sigma(A)$ is a Choquet simplex. Applied to Theorem~\ref{thm:extreme-factor}, this
gives the second proof of Theorem~\ref{thm:stationaries.form.simplex}.
For a general $\Gamma$-action on a separable $C^*$-algebra, we further show that $S^\Gamma(A)$ is a closed face of $S_\mu(A)$. At the opposite end, we
write $S^\Gamma(A)^\perp$ for the stationary states which dominate no nonzero positive
invariant functional. If $\nu_\varphi$ denotes the unique Choquet maximal representing
measure of \mbox{$\varphi\in S_\mu(A)$}, then
\[
    \varphi\in S^\Gamma(A)^\perp
    \quad\Longleftrightarrow\quad
    \nu_\varphi(\partial_eS^\Gamma(A))=0.
\]
In the language of split faces, this characterization identifies $S^\Gamma(A)^\perp$ as the
complementary face of $S^\Gamma(A)$, and yields the following canonical decomposition.

\begin{introthm}[Theorem~\ref{thm:trace-decomposition}]
Assume $S_{\mu}(A)$ is a Choquet simplex. Let \mbox{$\varphi\in S_\mu(A)$} be neither invariant nor purely noninvariant. Then there are unique
$\rho\in S^\Gamma(A)$, \mbox{$\psi\in S^\Gamma(A)^\perp$}, and $t\in(0,1)$ such that
\[
    \varphi=t\rho+(1-t)\psi.
\]
\end{introthm}
In the inner generating setting, $S^\Gamma(A)=T(A)$, and the decomposition above is the
tracial--purely nontracial decomposition.
The type of $M_{\varphi}$
gives
an operator-algebraic characterization of the two ends of the decomposition above.

\begin{introthm}[Theorem~\ref{thm:finite-typeIII}]
Let $\varphi\in S_{\mu}(A)$. Then
\begin{align*}
    \varphi\in T(A)
    &\quad\Longleftrightarrow\quad M_\varphi\text{ is finite},\\
    \varphi\text{ is purely nontracial}
    &\quad\Longleftrightarrow\quad M_\varphi\text{ is of type }\mathrm{III}.
\end{align*}
\end{introthm}

Combining Theorems~\ref{thm:extreme-factor} and~\ref{thm:finite-typeIII}, the GNS
algebra of an extreme stationary state is a finite factor when the state is tracial and
a type~$\mathrm{III}$ factor otherwise.

In the examples and applications section, we turn to the topology of the extreme boundary. We show that, for \mbox{$2\leq d\leq\infty$}, the stationary state simplex of $C^*(\mathbb{F}_d)$ is not Bauer, whereas for every nontrivial property~$(T)$ group $\Gamma$, the stationary state simplex of $C^*(\Gamma)$ is not Poulsen.
Lastly, we prove a sufficient condition for $S_{\mu}(A)$ to form a Bauer simplex.
\begin{introthm}[Theorem \ref{thm:bauer.simplex}]
        Assume there exists $\varepsilon>0$ and a finite set \mbox{$K\subset \mathrm{supp}(\check\mu)$} such that
    $$\max_{g\in K}\|\alpha_g(a)-a\|_{\varphi}\geq\varepsilon$$
    for every $\varphi\in \partial_eS_{\mu}(A)$ and every $a\in A_{sa}$ such that $\varphi(a)=0$ and $\|a\|_{\varphi}=1$. Then $S_{\mu}(A)$ is a Bauer simplex.
\end{introthm}

\section{Preliminaries}
\subsection{. Convexity and stationary states}
Throughout, $A$ is a unital separable $C^*$-algebra and $\Gamma$ is a countable
discrete group, unless stated otherwise. We denote the unitary group of $A$ by $U(A)$,
the state space of $A$ by $S(A)$, and its tracial state space by $T(A)$. State spaces
are equipped with the weak$^*$ topology. Since $A$ is separable, $S(A)$ and all of its
weak$^*$ closed subsets are metrizable.
For a metrizable, compact convex set $K$, we write $\partial_eK$ for its extreme boundary. The set
$K$ is a \emph{Choquet simplex} if every point of $K$ has a unique representing
probability measure concentrated on $\partial_eK$. It is a \emph{Bauer simplex} if
$\partial_eK$ is closed, and a nontrivial metrizable Choquet simplex is the
\emph{Poulsen simplex} if $\partial_eK$ is dense \cite{lindenstrauss1978poulsen}.
A convex subset $F\subset K$ is a
\emph{face} if
\[
    tx+(1-t)y\in F,\quad x,y\in K,\quad 0<t<1,
\]
implies $x,y\in F$. 
We write $\prob(\Gamma)$ for the probability measures on $\Gamma$. A measure
\mbox{$\mu\in\prob(\Gamma)$} is called \emph{generating} if $\mathrm{supp}(\mu)$ generates $\Gamma$
as a semigroup. We write
\[
    \check\mu(g)=\mu(g^{-1}),\qquad g\in\Gamma.
\]
Then $\check\mu$ is generating whenever $\mu$ is generating.
Suppose that $\Gamma$ acts on $A$ by a left action
\mbox{$\alpha\colon\Gamma\to\operatorname{Aut}(A)$}. We abbreviate $\alpha_g(a)$ by
$g\mathbin{\cdot}a$. The corresponding left action on $S(A)$ is
\[
    (g\mathbin{\cdot}\varphi)(a)
    =\varphi(\alpha_{g^{-1}}(a)),
    \qquad g\in\Gamma,\ \varphi\in S(A),\ a\in A.
\]
Its fixed-point space is denoted by
\[
    S^\Gamma(A)
    =\{\varphi\in S(A):g\mathbin{\cdot}\varphi=\varphi
      \text{ for every }g\in\Gamma\}.
\]
The Markov operator on $A$ associated with $\mu$ is
\[
    \Phi_\mu(a)
    =\sum_{g\in\Gamma}\mu(g)\alpha_{g^{-1}}(a)
    =\sum_{g\in\Gamma}\check\mu(g)\alpha_g(a).
\]
Both sums converge in norm. The map $\Phi_\mu$ is unital and completely positive. A
state \mbox{$\varphi\in S(A)$} is \emph{$\mu$-stationary} if
\[
    \varphi\circ\Phi_\mu=\varphi,
\]
or, equivalently,
\[
    \varphi=\sum_{g\in\Gamma}\mu(g)(g\mathbin{\cdot}\varphi).
\]
We denote the stationary state space by $S_\mu(A)$. It is a nonempty weak$^*$ compact
convex subset of $S(A)$: compactness is immediate, and nonemptiness follows by applying
the fixed-point theorem to the continuous affine self-map
\mbox{$\varphi\mapsto\varphi\circ\Phi_\mu$} of $S(A)$.
For $n\geq0$, let $\Phi_\mu^{(n)}$ denote the $n$th iterate of $\Phi_\mu$. With the
usual convention at $n=0$,
\[
    \Phi_\mu^{(n)}
    =\sum_{g\in\Gamma}\check\mu^{*n}(g)\alpha_g.
\]
If \mbox{$\varphi\in S_\mu(A)$}, then
\[
    \varphi\circ\Phi_\mu^{(n)}=\varphi
    \qquad\text{for every }n\geq0.
\]
For $a\in A$, the associated Cesaro averages are
\[
    a_N=\frac1N\sum_{n=0}^{N-1}\Phi_\mu^{(n)}(a).
\]
Thus \mbox{$\varphi(a_N)=\varphi(a)$} for every stationary state $\varphi$.
For the inner actions considered in most of the paper, \mbox{$\Gamma\leq U(A)$} and
\[
    \alpha_g(a)=u_gau_g^*,
\]
where $u_g$ denotes the unitary corresponding to $g\in\Gamma$. If $\Gamma$ generates
$A$ as a $C^*$-algebra, then
\[
    S^\Gamma(A)=T(A).
\]
Indeed, a trace is invariant under every inner automorphism. Conversely, invariance
gives \mbox{$\varphi(u_ga)=\varphi(au_g)$} for every $g\in\Gamma$ and $a\in A$, and the
density of the linear span of the $u_g$'s implies that $\varphi$ is tracial. In
particular,
\[
    T(A)\subset S_\mu(A).
\]

 We shall use the following consequence of Batty's characterization of invariant-state
simplices \cite[Theorem~6.1]{cjk1982invariant}.

\begin{prop}[Batty's criterion]\label{prop:batty.criterion}
Let $\Gamma$ act on $A$, let \mbox{$\mu\in\prob(\Gamma)$}, and let $a_N$ be the Cesaro
averages defined above. If
\[
    \varphi([a_N,b])\xrightarrow[N\to\infty]{}0
\]
for every $a,b\in A$ and every \mbox{$\varphi\in S_\mu(A)$}, then $S_\mu(A)$ is a
Choquet simplex.
\end{prop}
\subsection{. GNS representations and equivalence relations}
For \mbox{$\varphi\in S(A)$}, we write
$(\pi_\varphi,H_\varphi,\xi_\varphi)$ for its GNS representation and set
\[
    M_\varphi=\pi_\varphi(A)''.
\]
The vector state
\[
    \omega_\varphi(x)=\langle x\xi_\varphi,\xi_\varphi\rangle,
    \qquad x\in M_\varphi,
\]
is the canonical normal extension of $\varphi$ to $M_\varphi$.
For a von Neumann algebra $M$, we write \mbox{$Z(M)=M\cap M'$} for its center, and call $M$
a \emph{factor} when \mbox{$Z(M)=\mathbb{C}1$}.
\begin{dfn}[Equivalence of representations]
Let \mbox{$\pi\colon A\to B(H_\pi)$} and \mbox{$\rho\colon A\to B(H_\rho)$} be nondegenerate
representations.
\begin{enumerate}
    \item The representations $\pi$ and $\rho$ are \emph{unitarily equivalent}, written
    \mbox{$\pi\simeq\rho$}, if there is a unitary \mbox{$V\colon H_\pi\to H_\rho$} such that
    \[
        V\pi(a)=\rho(a)V,\qquad a\in A.
    \]
    \item The representation $\pi$ is \emph{quasi-contained} in $\rho$, written
    \mbox{$\pi\preceq\rho$}, if $\pi$ is unitarily equivalent to a subrepresentation of an
    amplification \mbox{$\rho^{(I)}=\bigoplus_{i\in I}\rho$} for some index set $I$.
    \item The representations $\pi$ and $\rho$ are \emph{quasi-equivalent}, written
    \mbox{$\pi\sim_{\mathrm{qe}}\rho$}, if \mbox{$\pi\preceq\rho$} and \mbox{$\rho\preceq\pi$}.
\end{enumerate}
\end{dfn}

\begin{prop}\label{prop:quasi.equivalence.characterization}
Two nondegenerate representations $\pi$ and $\rho$ are quasi-equivalent if and only if
the correspondence
\[
    \pi(a)\longmapsto\rho(a),\qquad a\in A,
\]
extends to a normal $*$-isomorphism from $\pi(A)''$ onto $\rho(A)''$.
\end{prop}

\begin{proof}
This is standard; see \cite[Section~III.5.1]{Blackadar2006}.
\end{proof}

We will also make use of the GNS Radon-Nikodym theorem due to Arveson \cite[Theorem 1.4.2]{arveson1969subalgebras}.
\begin{prop}[GNS Radon--Nikodym theorem]\label{prop:gns.radon.nikodym}
If $\varphi$ is a state and
$\psi$ is a positive functional satisfying \mbox{$0\leq\psi\leq c\varphi$}, then there is a
positive element \mbox{$h\in\pi_\varphi(A)'$}, with $0\leq h\leq c$, such that
\[
    \psi(a)
    =\langle\pi_\varphi(a)h\xi_\varphi,\xi_\varphi\rangle
    =\langle\pi_\varphi(a)h^{1/2}\xi_\varphi,
      h^{1/2}\xi_\varphi\rangle,
    \qquad a\in A.
\]
Consequently, $\psi$ extends to the normal positive functional
\[
    \widetilde\psi(x)
    =\langle xh^{1/2}\xi_\varphi,h^{1/2}\xi_\varphi\rangle,
    \qquad x\in M_\varphi.
\]
\end{prop}

The following definition is due to Batty \cite[Section 2]{cjk1982invariant}.
\begin{dfn}[$\Sigma$-equivalence]\label{def:sigma.equivalence}
Let $\Sigma$ be a semigroup of unital completely positive maps on $A$, and set
\[
    S_\Sigma(A)
    =\{\varphi\in S(A):\varphi\circ\sigma=\varphi
      \text{ for every }\sigma\in\Sigma\}.
\]
For \mbox{$\varphi\in S_\Sigma(A)$} and \mbox{$\sigma\in\Sigma$}, the formula
\[
    \widehat\sigma_\varphi(\pi_\varphi(a)\xi_\varphi)
    =\pi_\varphi(\sigma(a))\xi_\varphi
\]
defines a contraction on $H_\varphi$. Let
\[
    H_\varphi^\Sigma
    =\{\eta\in H_\varphi:\widehat\sigma_\varphi\eta=\eta
      \text{ for every }\sigma\in\Sigma\}.
\]
Two states \mbox{$\varphi,\psi\in S_\Sigma(A)$} are \emph{$\Sigma$-equivalent} if there is a
unitary \mbox{$V\colon H_\varphi\to H_\psi$} which intertwines $\pi_\varphi$ and $\pi_\psi$
and satisfies \mbox{$V(H_\varphi^\Sigma)=H_\psi^\Sigma$}.
\end{dfn}
In particular, $\Sigma$-equivalent states have
unitarily equivalent GNS representations.

\section{The simplex structure of stationary states}
Throughout this section and the following sections, $A$ is a separable $C^*$-algebra, $\Gamma$ is a discrete, countable group acting on $A$, and $\mu\in\prob(\Gamma)$ is generating.
\begin{lem}
    Let $a\in A$. For \mbox{$N\in\mathbb{N}$}, define 
    \mbox{$\displaystyle a_N=\frac{1}{N}\sum_{n=0}^{N-1}\Phi_{\mu}^{(n)}(a)$}. Then for every positive integer $m<N$, 
    $$\|a_N-\Phi_{\mu}^{(m)}(a_N)\|\leq \frac{2m}{N}\|a\|.$$
    \label{lem:bound.on.difference}
\end{lem}
\begin{proof}
    By definition, for $m<N$,
    $$\Phi_{\mu}^{(m)}(a_N)-a_N=\frac{1}{N}\left[\sum_{n=0}^{N-1}\Phi_{\mu}^{(m+n)}(a)-\sum_{n=0}^{N-1}\Phi_{\mu}^{(n)}(a)\right]=\frac{1}{N}\left[\sum_{n=N}^{N-1+m}\Phi_{\mu}^{(n)}(a)-\sum_{n=0}^{m-1}\Phi_{\mu}^{(n)}(a)\right].$$
    Since all the $\Phi_{\mu}^{(n)}$'s are contractions,
    $$\|a_N-\Phi_{\mu}^{(m)}(a_N)\|\leq \frac{1}{N}\left(\left\|\sum_{n=N}^{N-1+m}\Phi_{\mu}^{(n)}(a)\right\|+\left\|\sum_{n=0}^{m-1}\Phi_{\mu}^{(n)}(a)\right\|\right)\leq \frac{2m}{N}\|a\|.$$
\end{proof}
\begin{lem}
    Let \mbox{$\varphi\in S_{\mu}(A)$}, $g\in \Gamma$ and $a\in A$. Let $a_N$ be as in Lemma~\ref{lem:bound.on.difference}. Then
    $$\varphi((\alpha_g(a_N)-a_N)^*(\alpha_g(a_N)-a_N))\xrightarrow[N\to\infty]{}0$$
    \label{lem:difference.goes.to.0}
\end{lem}
\begin{proof}
    Fix \mbox{$\varphi\in S_{\mu}(A)$}, $a\in A$, $g\in \Gamma$ and take 
    \mbox{$m\in\mathbb{N}$} such that $\check{\mu}^{*m}(g)>0$. Then
    \begin{align*}
&\check{\mu}^{*m}(g)
 \varphi\!\left(
 (\alpha_g(a_N)-a_N)^*(\alpha_g(a_N)-a_N)
 \right) \\
&\leq
\sum_{h\in \Gamma}\check{\mu}^{*m}(h)
\varphi\!\left(
(\alpha_h(a_N)-a_N)^*(\alpha_h(a_N)-a_N)
\right) \\
&=
\sum_{h\in \Gamma}\check{\mu}^{*m}(h)
\varphi\!\left(
\alpha_h(a_N^*a_N)+a_N^*a_N
-2\operatorname{Re}\!\left(a_N^*\alpha_h(a_N)\right)
\right) \\
&=
2\operatorname{Re}
\varphi\!\left(a_N^*\bigl(a_N-\Phi_{\mu}^{(m)}(a_N)\bigr)\right).
\end{align*}
    where we used stationarity in the last step.
    For $N>m$, we get that 
    \begin{align*}
0\leq &\check{\mu}^{*m}(g)
 \varphi\!\left(
 (\alpha_g(a_N)-a_N)^*(\alpha_g(a_N)-a_N)
 \right) \\
&\qquad\leq
2\left|
\varphi\!\left(a_N^*(a_N-\Phi_{\mu}^{(m)}(a_N))\right)
\right| \\
&\qquad\leq
2\left\|a_N^*(a_N-\Phi_{\mu}^{(m)}(a_N))\right\| \\
&\qquad\leq
\frac{4m}{N}\|a\|^2
\xrightarrow[N\to\infty]{}0.
\end{align*}
Since $\check\mu^{*m}(g)>0$, the conclusion follows.
\end{proof}
\begin{thm}\label{thm:stationaries.form.simplex}
Let $A$ be a unital separable $C^*$-algebra, let \mbox{$\Gamma\leq U(A)$} be a countable
subgroup that generates $A$ as a $C^*$-algebra and acts on $A$ by conjugation, and let
\mbox{$\mu\in\prob(\Gamma)$} be a generating measure. Then $S_\mu(A)$ is a Choquet
simplex.
\end{thm}

\begin{proof}
   
    By Proposition~\ref{prop:batty.criterion}, it is enough to prove that for every $a,b\in A$ and \mbox{$\varphi\in S_{\mu}(A)$}, 
    \[\varphi([a_N,b])\xrightarrow[N\to\infty]{}0.\]
Let \mbox{$\varphi\in S_{\mu}(A)$} and $a\in A$. We first prove that for every $g\in \Gamma$,
\mbox{$\varphi([a_N,u_g])\xrightarrow[N\to\infty]{}0$}. Observe that
    \[
    [a_N,u_g]=a_Nu_g-u_ga_N
    =u_g\bigl(\alpha_{g^{-1}}(a_N)-a_N\bigr),
    \]
    and therefore by Lemma \ref{lem:difference.goes.to.0},
    \[
    |\varphi([a_N,u_g])|^2
    \leq \varphi([a_N,u_g]^*[a_N,u_g])
    =\varphi\!\left((\alpha_{g^{-1}}(a_N)-a_N)^*
      (\alpha_{g^{-1}}(a_N)-a_N)\right)
    \xrightarrow[N\to\infty]{}0.
    \]
 Now if \mbox{$\displaystyle b=\sum_{i=1}^n \alpha_iu_{g_i}$}, by the triangle inequality 
 \mbox{$\varphi([a_N,b])\xrightarrow[N\to\infty]{}0$}.
 Lastly, let $b\in A$ be arbitrary, and let $\varepsilon>0$. Take \mbox{$\displaystyle c=\sum_{i=1}^n\alpha_iu_{g_i}$} such that \mbox{$\|b-c\|<\frac{\varepsilon}{2\|a\|+1}$}. Then
\begin{align*}
0
&\leq \left|\varphi([a_N,b])\right| \\
&\leq \left|\varphi([a_N,b-c])\right|
   +\left|\varphi([a_N,c])\right| \\
&\leq 2\|a_N\|\,\|b-c\|
   +\left|\varphi([a_N,c])\right| \\
&\leq 2\|a\|\,\|b-c\|
   +\left|\varphi([a_N,c])\right| \\
&< \varepsilon+\left|\varphi([a_N,c])\right|.
\end{align*}
Taking $\limsup$ of both sides, we see that \mbox{$|\varphi([a_N,b])|\xrightarrow[N\to\infty]{}0$}.
\end{proof}

\begin{example}
Let $\Gamma$ be a countable discrete group, let \mbox{$\pi\colon\Gamma\to U(H)$} be a unitary
representation, and let
\[
 A=C^*_{\pi}(\Gamma)=C^*(\{\pi(g):g\in\Gamma\}).
\]
The group $\Gamma$ acts on $A$ by
\[
 \alpha_g(a)=\pi(g)a\pi(g)^*.
\]
By Theorem~\ref{thm:stationaries.form.simplex}, $S_\mu(A)$ is a Choquet simplex.
\end{example}
The following finite-dimensional example shows that the three hypotheses - namely that the action is inner, $\Gamma$ generates $A$ as a $C^*$-algebra and $\mu$ is generating - cannot, in general, be omitted.
\begin{example}
    Let $A=M_2(\mathbb{C})$. It is easy to verify that $S(A)$ is affinely homeomorphic to
    $$K=\{B\in M_
    2(\mathbb{C})_{\mathrm{sa}}\colon B\geq 0,\,\mathrm{tr}(B)=1\}$$
    via $B\to \varphi_B$, where $\varphi_B(C)=\mathrm{tr}(B^TC)$. The inverse is given by $\displaystyle\varphi\to (\varphi(E_{i,j}))_{i,j=1}^2$, with $(E_{i,j})_{i,j=1}^2$ being the matrix-units in $A$. Moreover,
    $$\partial_eK=\left\{P\in K:P \text{ is a  projection}\right\}.$$
    In particular, $K$ is not a Choquet simplex, since
    $$\frac{1}{2}I_2=\frac{1}{2}\begin{pmatrix}
        1&0
        \\
        0&0
    \end{pmatrix}+\frac{1}{2}\begin{pmatrix}
        0&0
        \\
        0&1
    \end{pmatrix}=\frac{1}{2}\begin{pmatrix}
        \frac{1}{2}&\frac{1}{2}\\\frac{1}{2}&\frac{1}{2}
    \end{pmatrix}+\frac{1}{2}\begin{pmatrix}
        \frac{1}{2}&-\frac{1}{2}\\-\frac{1}{2}&\frac{1}{2}
    \end{pmatrix},$$
    that is, $\frac{1}{2}I_2$ can be written as two distinct convex combinations of the extreme points of $K$. 
    \begin{itemize}
        \item If $\Gamma=\langle u,v\rangle$ with $u,v$ a pair of unitaries that generate $A$, and the action of $\Gamma$ on $A$ is the trivial action, then $S_{\mu}(A)=S(A)$ for every generating probability measure $\mu\in\prob(\Gamma)$. In particular, it is not a Choquet simplex.
        \item  If $\Gamma=\{I_2\}$ and $\Gamma$ acts on $A$ by inner automorphisms, $S_{\mu}(A)=S(A)$ for every generating probability measure $\mu\in\prob(\Gamma)$, and again $S_{\mu}(A)$ is not a Choquet simplex.
        \item If $\Gamma=\langle u,v\rangle$ for a pair of unitaries that generate $A$, and the action on $A$ is by inner automorphisms, taking $\mu=\delta_{I_2}\in\prob(\Gamma)$, gives us $S_{\mu}(A)=S(A)$, and therefore $S_{\mu}(A)$ is not a Choquet simplex.
    \end{itemize}
\end{example}
It can be the case, however, that the stationaries form a simplex for other group actions. The following is inspired by \cite[Corollary 4.2]{slutsky2026invariant}.
\begin{thm}
    Let $I$ be a countable set with a $\Gamma$-action such that the following conditions hold:
    \begin{enumerate}
        \item There are $C^*$-subalgebras $A_F\subset A$ indexed by finite subsets of $I$ such that
        \[
        A_{\mathrm{loc}}=\bigcup_{F\subset I\text{ finite}}A_F
        \]
        is a norm-dense $*$-subalgebra of $A$.
        \item Every $\Gamma$-orbit of $I$ is infinite.
        \item \mbox{$\alpha_g(A_F)=A_{gF}$}.
        \item \mbox{$F\cap K=\emptyset\Rightarrow [A_F,A_K]=\{0\}$}.
    \end{enumerate}
    Then $S_{\mu}(A)$ is a Choquet simplex.
    \label{thm:Choquet.simplex.finite.sets.}
\end{thm}
\begin{proof}
    Again using Proposition~\ref{prop:batty.criterion}, it is enough to prove that
    \[
        \|[a_N,b]\|\xrightarrow[N\to\infty]{}0
        \qquad\text{for every }a,b\in A.
    \]
    First, fix \mbox{$a,b\in A_{\mathrm{loc}}$} and take finite sets $F,K\subset I$ such that
    $a\in A_F$ and $b\in A_K$. Define
    $$B_{F,K}=\{g\in \Gamma\colon gF\cap K\neq\emptyset\}\subset \Gamma.$$
 Since \(\alpha_g(a)\) and \(b\) commute whenever \mbox{\(gF\cap K=\varnothing\)},
    we have, for every \(n\in\mathbb N\),
    \begin{align*}
        \lVert[\Phi_\mu^{(n)}(a),b]\rVert
        &= \left\lVert
            \sum_{g\in\Gamma}
            \check\mu^{*n}(g)[\alpha_g(a),b]
        \right\rVert                                                     \\
        &= \left\lVert
            \sum_{g\in B_{F,K}}
            \check\mu^{*n}(g)[\alpha_g(a),b]
        \right\rVert                                                     \\
        &\leq
        \sum_{g\in B_{F,K}}
        \check\mu^{*n}(g)\,
        \lVert[\alpha_g(a),b]\rVert                                      \\
        &\leq
        2\lVert a\rVert\lVert b\rVert\,
        \check\mu^{*n}(B_{F,K}).
    \end{align*}
    Consequently,
    \begin{equation}\label{eq:commutator-bound}
        \lVert[a_N,b]\rVert
        \leq
        2\lVert a\rVert\lVert b\rVert
        \frac{1}{N}\sum_{n=0}^{N-1}\check\mu^{*n}(B_{F,K}).
    \end{equation}
    It therefore suffices to prove that
    \[
        \frac{1}{N}\sum_{n=0}^{N-1}\check\mu^{*n}(B_{F,K})
        \xrightarrow[N\to\infty]{}0.
    \]    
 For this, let $\ell^2(I)$ be the Hilbert space of square-summable complex-valued sequences indexed by $I$, and let $e_i$ denote the element satisfying \mbox{$(e_i)_j=\delta_{i,j}$}. Observe that $g\in B_{F,K}$ if and only if \mbox{$\langle e_{g.i},e_j\rangle\neq 0$} for some $i\in F,\,j\in K$. Therefore,
    \[
    \check\mu^{*n}(B_{F,K})
    =\sum_{g\in B_{F,K}}\check\mu^{*n}(g)
    \leq \sum_{\substack{i\in F\\j\in K}}\sum_{g\in\Gamma}
       \check\mu^{*n}(g)\langle e_{gi},e_j\rangle.
    \]
    Let $\beta_g$ be the unitary on $\ell^2(I)$ defined by \mbox{$\beta_g(e_i)=e_{gi}$}, and define the contraction $Q$ on $\ell^2(I)$ by
    \[
    Q=\sum_{g\in\Gamma}\check\mu(g)\beta_g.
    \]
    It is straightforward to verify that for every \mbox{$n\in\mathbb{N}$},
    \[Q^n=\sum_{g\in\Gamma}\check\mu^{*n}(g)\beta_g,\]
    and thus by definition
    \[\check\mu^{*n}(B_{F,K})\leq
      \sum_{\substack{i\in F\\ j\in K}}\langle Q^n e_i,e_j\rangle.\]
    In particular, 
    \begin{align*}
        \frac{1}{N}\sum_{n=0}^{N-1}\check\mu^{*n}(B_{F,K})
        &\leq
        \sum_{\substack{i\in F\\ j\in K}}
        \langle \frac{1}{N}\sum_{n=0}^{N-1}Q^ne_i,e_j\rangle                                    \\
        &\leq
        |K||F|\max_{i\in F}\left\|\frac{1}{N}\sum_{n=0}^{N-1}Q^ne_i\right\|.
    \end{align*}
    By the mean ergodic theorem, for every $i\in I$,
    \[
        \frac{1}{N}\sum_{n=0}^{N-1}Q^ne_i
        \xrightarrow[N\to\infty]{}
        P_{\operatorname{Fix}(Q)}e_i
    \]
    where \(P_{\operatorname{Fix}(Q)}\) is the orthogonal projection onto
    \[
        \operatorname{Fix}(Q)
        \coloneqq
        \{\xi\in\ell^2(I):Q\xi=\xi\}.
    \]
Lastly, we prove that $Fix(Q)=\{0\}$. Observe that, since every $\beta_g$ is unitary, if $\xi$ is a fixed point of $Q$, then
    \begin{align*}
        \sum_{g\in\Gamma}
        \check\mu(g)\lVert\beta_g\xi-\xi\rVert^2
        &=
        \sum_{g\in\Gamma}\check\mu(g)
        \left(
            \lVert\beta_g\xi\rVert^2
            -2\operatorname{Re}\langle\beta_g\xi,\xi\rangle
            +\lVert\xi\rVert^2
        \right)                                                         \\
        &=
        2\lVert\xi\rVert^2
        -2\operatorname{Re}
        \left\langle
            \sum_{g\in\Gamma}\check\mu(g)\beta_g\xi,\xi
        \right\rangle                                                   \\
        &=
        2\lVert\xi\rVert^2
        -2\operatorname{Re}\langle Q\xi,\xi\rangle                       \\
        &=0.
    \end{align*}
    and therefore \mbox{$\beta_g(\xi)=\xi$} for every \mbox{$g\in\mathrm{supp}(\check\mu)$}. Since $\mu$ is generating, $\xi$ is a fixed point of every $\beta_g$.
    Now let $\xi\in\mathrm{Fix}(Q)$ and assume that there exists $j\in I$ such that \mbox{$\xi_j\neq 0$}. Since $\xi$ is fixed by $\beta_g$ for all $g\in \Gamma$, \mbox{$\xi_{gj}=\xi_j$} for every $g\in \Gamma$. Since the orbit $\Gamma j$ is infinite by assumption, this implies that $\xi$ has infinitely many coordinates equal to \mbox{$\xi_j\neq 0$}, and thus
    \[
    \sum_{i\in I}|\xi_i|^2\geq \sum_{i\in\Gamma j}|\xi_i|^2=\infty,
    \]
    contradicting the fact that \mbox{$\xi\in\ell^2(I)$}. Therefore $\mathrm{Fix}(Q)=\{0\}$, and
    \[
    \frac{1}{N}\sum_{n=0}^{N-1}\check\mu^{*n}(B_{F,K})
    \xrightarrow[N\to\infty]{}0,
    \]
concluding the first part. A standard density argument gives the result for every $a,b\in A$. Therefore, for every \mbox{$\varphi\in S_{\mu}(A)$} and every $a,b\in A$,
$$\varphi([a_N,b])\xrightarrow[N\to\infty]{}0,$$
and $S_{\mu}(A)$ is a Choquet simplex.
\end{proof}
\begin{example}
    Let $D$ be a separable unital $C^*$-algebra, let $\Gamma$ be an infinite countable discrete group, let $I=\Gamma$ with the left regular action, and let
    \[
    A=\bigotimes_{i\in\Gamma}^{\min}D
    \]
    be the spatial infinite tensor product with respect to the unit of $D$. Let $\Gamma$ act on $A$ via the induced Bernoulli shift. For a finite set $F\subset I$, let $A_F$ be the copy of \mbox{$\bigotimes_{i\in F}^{\min}D$} inside $A$. Then the algebras $A_F$ and the action on $A$ satisfy the conditions of Theorem~\ref{thm:Choquet.simplex.finite.sets.}, and therefore $S_{\mu}(A)$ is a Choquet simplex.
\end{example}
\section{Extreme stationary states and factor representations}
Assume now that $\Gamma\leq U(A)$ is a discrete subgroup generating $A$ as a $C^*$ algebra and acts on $A$ by conjugation. Fix
$\varphi\in S_\mu(A)$. Recall that \mbox{$\alpha_g(a)=u_gau_g^*$}, and put
\[
    U_g=\pi_\varphi(u_g),\qquad g\in\Gamma.
\]
The canonical normal extension of $\alpha_g$ to $M_\varphi$ is
\[
    \widetilde\alpha_g=\operatorname{Ad}(U_g),
    \qquad
    \widetilde\alpha_g(x)=U_gxU_g^*,
    \qquad x\in M_\varphi.
\]
Indeed,
\[
    \widetilde\alpha_g(\pi_\varphi(a))
    =U_g\pi_\varphi(a)U_g^*
    =\pi_\varphi(u_gau_g^*)
    =\pi_\varphi(\alpha_g(a)),
    \qquad a\in A.
\]
Thus \mbox{$(\widetilde\alpha_g)_{g\in\Gamma}$} is an action by normal automorphisms of
$M_\varphi$ extending the given action on $A$. 
Corresponding to the convention
\mbox{$\Phi_\mu=\sum_{g\in\Gamma}\mu(g)\alpha_{g^{-1}}$}, define
\[
    T_\varphi(x)
    =\sum_{g\in\Gamma}\mu(g)\widetilde\alpha_{g^{-1}}(x)
    =\sum_{g\in\Gamma}\mu(g)U_g^*xU_g
    =\sum_{g\in\Gamma}\check\mu(g)\widetilde\alpha_g(x).
\]
The sums converge in norm for every $x\in M_\varphi$, and $T_\varphi$ is a normal
unital completely positive map satisfying
\[
    T_\varphi(\pi_\varphi(a))
    =\pi_\varphi(\Phi_\mu(a)),
    \qquad a\in A.
\]
\begin{prop}\label{prop:normal.stationarity}
The canonical normal extension $\omega_\varphi$ of $\varphi$ is $T_\varphi$-stationary;
that is,
\[
    \omega_\varphi\circ T_\varphi=\omega_\varphi.
\]
\end{prop}

\begin{proof}
For every $a\in A$,
\[
    (\omega_\varphi\circ T_\varphi)(\pi_\varphi(a))
    =\varphi(\Phi_\mu(a))
    =\varphi(a)
    =\omega_\varphi(\pi_\varphi(a)).
\]
Both functionals are normal, so the conclusion follows from the ultraweak density of
$\pi_\varphi(A)$ in $M_\varphi$.
\end{proof}
We write
\[
    \operatorname{Fix}(T_\varphi)
    =\{x\in M_\varphi:T_\varphi(x)=x\}
\]
and
\[
    M_\varphi^\Gamma
    =\{x\in M_\varphi:\widetilde\alpha_g(x)=x
      \text{ for every }g\in\Gamma\}.
\]
Finally, let
\[
    E_N=\frac1N\sum_{n=0}^{N-1}T_\varphi^n.
\]
We use the mean ergodic result in \cite[Theorem~3.3]{Izumi-boundaries}: a subnet
of $(E_N)$ converges in the point-weak operator topology to a positive unital
projection
\[
    E\colon M_\varphi\longrightarrow\operatorname{Fix}(T_\varphi).
\]
We fix such a subnet and denote its limit by $E$. Since
\mbox{$\omega_\varphi\circ E_N=\omega_\varphi$} for every $N$, one also has
\mbox{$\omega_\varphi\circ E=\omega_\varphi$}.
\begin{lem}\label{lem:separating}
The GNS vector $\xi_{\varphi}$ is separating for $M_{\varphi}$. Equivalently, $\omega_{\varphi}$ is a
faithful normal state on $M_{\varphi}$.
\label{lem:faithful.normal.state.on.vN}
\end{lem}

\begin{proof}
Suppose that \mbox{$x\in M_{\varphi}$} satisfies \mbox{$x\xi_{\varphi}=0$}. By definition, \mbox{$\omega_{\varphi}(x^*x)=0$}.
For every $n\geq 1$, stationarity gives
\begin{align*}
    0
    &=\omega_{\varphi}(x^*x)\\
    &=\sum_{h\in \Gamma}\mu^{*n}(h)
      \omega_{\varphi}(\tilde\alpha_{h^{-1}}(x^*x))\\
    &=\sum_{h\in \Gamma}\mu^{*n}(h)\|xU_h\xi_{\varphi}\|^2.
\end{align*}
Since all the summands are nonnegative, \mbox{$xU_h\xi_{\varphi}=0$}
whenever $\mu^{*n}(h)>0$. By our assumption, $\operatorname{supp}\mu$ generates
$\Gamma$ as a semigroup, and thus \mbox{$xU_g\xi_{\varphi}=0$} for every $g\in \Gamma$.
Moreover,
\[
    \overline{\operatorname{span}}
    \{U_h\xi_{\varphi}:h\in \Gamma\}=H_{\varphi},
\]
because $\Gamma$ generates $A$ and $\xi_{\varphi}$ is cyclic for $\pi_{\varphi}(A)$. Thus $x=0$,
so $\xi_{\varphi}$ is separating for $M_{\varphi}$.
\end{proof}

\begin{lem}\label{lem:fixed-points}
Let $\mathrm{Fix}(T_{\varphi})$ denote the fixed-point space of $T_{\varphi}$. Then \mbox{$\mathrm{Fix}(T_{\varphi})=Z(M_{\varphi})$}.
\end{lem}
\begin{proof}
We first show that \mbox{$\mathrm{Fix}(T_{\varphi})=M_{\varphi}^\Gamma$},
where
\[
    M_{\varphi}^\Gamma
    =\{x\in M_{\varphi}:\tilde{\alpha}_g(x)=x\text{ for every }g\in \Gamma\}.
\]
Let \mbox{$x\in M_{\varphi}$} satisfy $T_{\varphi}(x)=x$. Using \mbox{$\omega_{\varphi}\circ T_{\varphi}=\omega_{\varphi}$}, we obtain
\begin{align*}
    \sum_{g\in \Gamma}\check\mu(g)
       \omega_{\varphi}((\tilde{\alpha}_g(x)-x)^*(\tilde{\alpha}_g(x)-x))
    &={}
      \omega_{\varphi}(T_{\varphi}(x^*x))
      -\omega_{\varphi}(T_{\varphi}(x^*)x)
      -\omega_{\varphi}(x^*T_{\varphi}(x))
      +\omega_{\varphi}(x^*x)\\
    &=0.
\end{align*}
It follows from the faithfulness of $\omega_{\varphi}$ that \mbox{$\tilde{\alpha}_g(x)=x$} for every \mbox{$g\in\operatorname{supp}(\check\mu)$}.
Since $\check\mu$ is also generating, this yields \mbox{$\tilde{\alpha}_g(x)=x$} for
every $g\in \Gamma$. Thus \mbox{$x\in M_{\varphi}^\Gamma$}. The reverse inclusion is
immediate, and hence \mbox{$\mathrm{Fix}(T_{\varphi})=M_{\varphi}^\Gamma$}. Finally, the action is inner and
\mbox{$\{U_g:g\in \Gamma\}''=\pi_{\varphi}(A)''=M_{\varphi}$}.
Therefore
\[
    M_{\varphi}^\Gamma
    =M_{\varphi}\cap\{U_g:g\in \Gamma\}'
    =M_{\varphi}\cap M_{\varphi}'
    =Z(M_{\varphi}),
\]
as claimed.
\end{proof}

The implication from uniqueness of a faithful normal stationary state to factoriality
appears in \cite[Proposition~4.14]{hartman2023stationary}. After passing from $A$ to
$\pi_\varphi(A)$, the implication from extremality to ergodicity follows from
\cite[Proposition~2.7(3)]{bader2022charmenability}, while
\cite[Proposition~2.7(2)]{bader2022charmenability} gives uniqueness of the normal
stationary state in the ergodic case. In the present
inner setting ergodicity is equivalent to factoriality by
Lemma~\ref{lem:fixed-points}. We include a direct proof of the complete statement needed
here.

\begin{thm}\label{thm:extreme-factor}
Fix
\mbox{$\varphi\in S_\mu(A)$}. The following are equivalent:
\begin{enumerate}
    \item $\varphi$ is an extreme point of $S_\mu(A)$;
    \item $M_\varphi$ is a factor;
    \item \mbox{$\mathrm{Fix}(T_\varphi)=\mathbb{C}1$};
    \item $\omega_\varphi$ is the unique normal $\mu$-stationary state on $M_\varphi$.
\end{enumerate}
\end{thm}

\begin{proof}
The equivalence of (2) and (3) follows immediately from
Lemma~\ref{lem:fixed-points}. 
Assume that $M_{\varphi}$ is a factor and let \mbox{$x\in M_{\varphi}$}. $E(x)$ is a fixed point of $T_{\varphi}$, and thus $E(x)=c_x1$, with \mbox{$c_x\in\mathbb{C}$}. Therefore,
$$c_x=(\omega_{\varphi}\circ E)(x)=\omega_{\varphi}(x),$$
that is \mbox{$E(x)=\omega_{\varphi}(x)1$}.
Now let $\rho$ be a normal $\mu$-stationary state on $M_{\varphi}$. By stationarity, \mbox{$\rho=\rho\circ E_N$} for every \mbox{$N\in\mathbb{N}$}, and thus by normality 
\mbox{$\rho=\rho\circ E$}.
Hence for every \mbox{$x\in M_{\varphi}$},
$$\rho(x)=\rho(E(x))=\omega_{\varphi}(x),$$
and uniqueness is proved, proving (2)$\Rightarrow$(4).
We next prove (4)$\Rightarrow$(1). Suppose that
\[
    \varphi=t\varphi_1+(1-t)\varphi_2
\]
where \mbox{$\varphi_1,\varphi_2\in S_\mu(A)$} and $t\in (0,1)$. Since
\mbox{$t\varphi_1\leq\varphi$} and \mbox{$(1-t)\varphi_2\leq\varphi$},
the GNS Radon--Nikodym theorem
gives positive contractions
\mbox{$h_1,h_2\in\pi_{\varphi}(A)'$} such that
\begin{align*}
    t\varphi_1(a)
    &=\langle\pi_{\varphi}(a)h_1\xi_{\varphi},\xi_{\varphi}\rangle,\\
    (1-t)\varphi_2(a)
    &=\langle\pi_{\varphi}(a)h_2\xi_{\varphi},\xi_{\varphi}\rangle
\end{align*}
for every $a\in A$. Consequently, the formulas
\begin{align*}
    \omega_1(x)
    &=\frac{1}{t}
      \langle xh_1^{1/2}\xi_{\varphi},h_1^{1/2}\xi_{\varphi}\rangle,\\
    \omega_2(x)
    &=\frac{1}{1-t}
      \langle xh_2^{1/2}\xi_{\varphi},h_2^{1/2}\xi_{\varphi}\rangle
\end{align*}
define normal states on $M_{\varphi}$ extending $\varphi_1$ and $\varphi_2$,
respectively. Since $\omega_i\circ T_{\varphi}$ and $\omega_i$ agree on the ultraweakly dense
subalgebra $\pi_\varphi(A)$, each $\omega_i$ is $\mu$-stationary. By (4),
\mbox{$\omega_1=\omega_2=\omega_{\varphi}$}.
Restricting to $\pi_{\varphi}(A)$ gives
\mbox{$\varphi_1=\varphi_2=\varphi$}, and thus $\varphi$ is extreme.
\noindent
Finally we prove $(1)\Rightarrow (2)$. Assume that $M_{\varphi}$ is not a factor and take a nontrivial central
projection \mbox{$z\in Z(M_{\varphi})$}. By Lemma~\ref{lem:separating}, $\omega_{\varphi}$ is
faithful, so
\[
    0<\lambda:=\omega_{\varphi}(z)<1.
\]
Define normal states on $M_{\varphi}$ by
\[
    \omega_z(x)=\frac{\omega_{\varphi}(zx)}{\lambda},
    \qquad
    \omega_{1-z}(x)=\frac{\omega_{\varphi}((1-z)x)}{1-\lambda}.
\]
Every $\tilde{\alpha}_g$ fixes $z$, and hence
\[
    \omega_z\circ T_{\varphi}=\omega_z,
    \qquad
    \omega_{1-z}\circ T_{\varphi}=\omega_{1-z}.
\]
Their restrictions $\varphi_z$ and $\varphi_{1-z}$ to $\pi_{\varphi}(A)$ are
distinct $\mu$-stationary states. Indeed, equality of the restrictions would,
by normality and ultraweak density, imply equality of $\omega_z$ and
$\omega_{1-z}$, which is impossible since their support projections are
orthogonal. We therefore obtain the nontrivial decomposition
\[
    \varphi
    =\lambda\varphi_z+(1-\lambda)\varphi_{1-z}.
\]
Thus $\varphi$ is not extreme.
\end{proof}
The following Proposition is a partial converse for \ref{thm:extreme-factor}, and requires no action.
\begin{prop}
    Let \(\Sigma\) be a semigroup of unital completely positive maps on a separable \(C^*\)-algebra \(A\).
    Assume that, for every \(\varphi\in S_{\Sigma}(A)\),
    \[
    \varphi \in \partial_e S_{\Sigma}(A)
    \quad\Longleftrightarrow\quad
    M_{\varphi} \text{ is a factor}.
\]
    Then \(S_{\Sigma}(A)\) is a Choquet simplex.
    \label{prop:sufficient.cond.for.simplex}
\end{prop}
Before proving Proposition~\ref{prop:sufficient.cond.for.simplex}, we prove a technical lemma.

\begin{lem}
    Let \mbox{$\rho,\varphi_1,\varphi_2\in S(A)$} and assume there are $\lambda_1,\lambda_2>0$ such that
    \[
    \rho=\lambda_1\varphi_1+\lambda_2\varphi_2.
    \]
    Then
    \[
    \pi_{\rho}\sim_{\mathrm{qe}}\pi_{\varphi_1}\oplus\pi_{\varphi_2}.
    \]
    \label{lem:quasi.equiv.of.GNS.for.convex.comb}
\end{lem}
\begin{proof}
    Define the map
    \mbox{$U\colon \pi_{\rho}(A)\xi_{\rho}\to H_{\varphi_1}\oplus H_{\varphi_2}$} by $$U(\pi_{\rho}(a)\xi_{\rho})= \sqrt{\lambda_1}\pi_{\varphi_1}(a)\xi_{\varphi_1}\oplus \sqrt{\lambda_2}\pi_{\varphi_2}(a)\xi_{\varphi_2}.$$
    $U$ is an isometry on $\pi_{\rho}(A)\xi_{\rho}$ and thus extends to an isometry on $H_{\rho}$. Moreover, $U$ intertwines $\pi_{\rho}$ and \mbox{$\pi_{\varphi_1}\oplus\pi_{\varphi_2}$}, and therefore
    $$\pi_{\rho}\simeq \left(\pi_{\varphi_1}\oplus\pi_{\varphi_2}\right)|_{\overline{(\pi_{\varphi_1}\oplus\pi_{\varphi_2})(A)\left(\sqrt{\lambda_1}\xi_{\varphi_1}\oplus \sqrt{\lambda_2}\xi_{\varphi_2}\right)}}.$$
    On the other hand, \mbox{$\lambda_i\varphi_i\leq \rho$}, and thus by the GNS Radon-Nikodym derivative theorem, there are contractions \mbox{$h_1,\,h_2\in \pi_{\rho}(A)'$} such that
    $$\varphi_i(a)=\langle \pi_{\rho}(a)\zeta_i,\zeta_i\rangle,$$
    with \mbox{$\zeta_i=\lambda_i^{-\frac{1}{2}}h_i^{\frac{1}{2}}\xi_{\rho}$}. Letting \mbox{$K_i=\overline{\pi_{\rho}(A)\zeta_i}$}, the triple $(\pi_{\rho}|_{K_i}, K_i,\zeta_i)$ is a GNS representation for $\varphi_i$. By uniqueness, \mbox{$\pi_{\varphi_i}\simeq \pi_{\rho}|_{K_i}$}, and therefore \mbox{$\pi_{\varphi_1}\oplus \pi_{\varphi_2}$} is unitarily equivalent to a subrepresentation of \mbox{$\pi_{\rho}^{\oplus 2}$}. Since \mbox{$\pi_{\rho}\sim_{\mathrm{qe}}\pi_{\rho}^{\oplus 2}$},
    we get the chain
    \[
    \pi_{\rho}\preceq\pi_{\varphi_1}\oplus \pi_{\varphi_2}
    \preceq \pi_{\rho}\oplus \pi_{\rho}\sim_{\mathrm{qe}}\pi_{\rho},
    \]
    and the conclusion follows.
\end{proof}
\begin{proof}[Proof of Proposition \ref{prop:sufficient.cond.for.simplex}]
    Assume $S_{\Sigma}(A)$ is not a simplex. By \cite[Theorem~6.1]{cjk1982invariant} there are distinct $\Sigma$-invariant states, $\varphi_1,\varphi_2\in\partial_eS_{\Sigma}(A)$, which are $\Sigma$-equivalent. In particular, $\pi_{\varphi_1}\simeq\pi_{\varphi_2}$.

    \noindent Set \mbox{$\rho=\frac{\varphi_1+\varphi_2}{2}\in S_{\Sigma}(A)$}. By Lemma~\ref{lem:quasi.equiv.of.GNS.for.convex.comb},
    \[
    \pi_{\rho}\sim_{\mathrm{qe}}\pi_{\varphi_1}\oplus\pi_{\varphi_2}
    \simeq \pi_{\varphi_1}\oplus \pi_{\varphi_1}
    \sim_{\mathrm{qe}}\pi_{\varphi_1}.
    \]
    In particular, \mbox{$M_{\rho}\cong M_{\varphi_1}$} \cite[Section~III.5.1]{Blackadar2006}, and thus $M_{\rho}$ is a factor. By assumption, this implies \mbox{$\rho\in \partial_eS_{\Sigma}(A)$}, and \mbox{$\rho=\varphi_1=\varphi_2$}, a contradiction.
\end{proof}
\begin{remark}
    Combining Theorem~\ref{thm:extreme-factor} and Proposition~\ref{prop:sufficient.cond.for.simplex}, we obtain a second proof of Theorem~\ref{thm:stationaries.form.simplex}.
\end{remark}

\begin{remark}
    The other direction of Proposition~\ref{prop:sufficient.cond.for.simplex} does not hold in general. Indeed, let $\mathbb{Z}$ act on $C(\mathbb{T})$ by an irrational rotation. This action has a unique invariant probability measure, namely the normalized Lebesgue measure $\lambda$. In particular, it is an extreme point of $\prob^{\mathbb Z}(\mathbb{T})$. However, its GNS representation is given by multiplication operators,
    \[
    \pi_{\lambda}\colon C(\mathbb{T})\to B(L^2(\mathbb{T},\lambda)),
    \qquad (\pi_{\lambda}(f)\zeta)(z)=f(z)\zeta(z),
    \]
    and \mbox{$M_{\lambda}=L^{\infty}(\mathbb{T},\lambda)$}, which is not a factor.
\end{remark}

\section{The invariant face and the type decomposition}
The following results concern general $\Gamma$ actions on a unital $C^*$-algebra $A$, with $A$ separable and $\Gamma$ discrete and countable.
\begin{prop}
 $S^{\Gamma}(A)\subset S_{\mu}(A)$ is a closed face.
    \label{prop:closed.face}
\end{prop}
\begin{proof}
Closedness is clear. Let $\tau\in S^{\Gamma}(A)$, let $t\in(0,1)$, and let
\mbox{$\varphi_1,\varphi_2\in S_\mu(A)$} be such that
\[
    \tau=t\varphi_1+(1-t)\varphi_2.
\]
Let $(\pi,\mathcal{H},\xi)$ be $\tau$'s GNS representation. Since \mbox{$t\varphi_1\leq \tau$}, the GNS Radon--Nikodym theorem \cite[Theorem~1.4.2]{arveson1969subalgebras} gives a positive contraction $h\in \pi(A)'$ such that
$$\varphi_1(a)=\frac{1}{t}\langle \pi(a)h\xi,\xi\rangle.$$
Letting \mbox{$\eta=\frac{1}{t}h\xi$}, and since $h\in \pi(A)'$ is self-adjoint, 
$$\varphi_1(a)=\langle \pi(a)\xi,\eta\rangle.$$
By stationarity,
\begin{align*}
    \varphi_1(a)
    &=\sum_{g\in \Gamma}\mu(g)\varphi_1(g^{-1}\mathbin{\cdot}a) \\
    &=\sum_{g\in \Gamma}\mu(g)      \langle\pi(g^{-1}\mathbin{\cdot}a)\xi,\eta\rangle \\
    &=\sum_{g\in \Gamma}\mu(g)
    \langle\rho(g)^*\pi(a)\xi,\eta\rangle \\
    &=\left\langle\pi(a)\xi,
     \sum_{g\in \Gamma}\mu(g)\rho(g)\eta\right\rangle,
\end{align*}
where \mbox{$\rho\colon \Gamma\to U(\mathcal H)$} is the unitary representation corresponding
to $\tau$ (\!\cite[Theorem 5.3]{segal1951class}, \cite[Lemma 12.6]{KennedyS}). Since \mbox{$\pi(A)\xi\subset \mathcal{H}$} is a dense subspace,
\[
    \eta=\sum_{g\in \Gamma}\mu(g)\rho(g)\eta.
\]
Since the $\rho(g)$'s are unitaries and

\begin{align*}
\sum_{g\in \Gamma}\mu(g)\lVert\rho(g)\eta-\eta\rVert^2
&=
\sum_{g\in \Gamma}\mu(g)
\left(
\lVert\rho(g)\eta\rVert^2
-2\operatorname{Re}\langle\rho(g)\eta,\eta\rangle
+\lVert\eta\rVert^2
\right) \\
&=
\sum_{g\in \Gamma}\mu(g)\lVert\rho(g)\eta\rVert^2
-2\lVert\eta\rVert^2
+\lVert\eta\rVert^2 \\
&=0,
\end{align*}
\mbox{$\rho(g)\eta=\eta$} for every \mbox{$g\in\operatorname{supp}\mu$}, and hence for
every $g\in \Gamma$, since $\mu$ is generating. Therefore,
\begin{align*}
    (g\mathbin{\cdot}\varphi_1)(a)
&=\langle\pi(g^{-1}\mathbin{\cdot}a)\xi,\eta\rangle \\
    &=\langle\pi(a)\xi,\rho(g)\eta\rangle
      =\langle\pi(a)\xi,\eta\rangle
      =\varphi_1(a).
\end{align*}
Thus \mbox{$\varphi_1\in S^\Gamma(A)$}. A similar argument shows
\mbox{$\varphi_2\in S^\Gamma(A)$}.
\end{proof}

\begin{lem}\label{prop:purely.nontracial.iff.measure.0}
Assume $S_{\mu}(A)$ is a Choquet simplex. For \mbox{$\varphi\in S_\mu(A)$}, let $\nu_\varphi$ be the unique Choquet maximal
representing probability measure for $\varphi$ on $S_\mu(A)$. Then
\[
    \varphi\in S^{\Gamma}(A)^\perp
    \quad\Longleftrightarrow\quad
    \nu_\varphi(\partial_eS^{\Gamma}(A))=0.
\]
\end{lem}

\begin{proof}
Assume that $\varphi$ is not purely noninvariant, and let \mbox{$0\leq\tau\leq\varphi$} be a
nonzero positive invariant functional. After replacing $\tau$ by a positive scalar multiple
if necessary, we may assume that \mbox{$t=\tau(1)\in(0,1)$}. Write
\[
    \varphi=t\rho+(1-t)\psi,
    \qquad
    \rho=\frac{1}{t}\tau,
    \qquad
    \psi=\frac{1}{1-t}(\varphi-\tau)\in S_\mu(A).
\]
Since $t\nu_\rho+(1-t)\nu_\psi$ is a Choquet maximal representing measure for
$\varphi$, uniqueness gives
\[
    \nu_\varphi=t\nu_\rho+(1-t)\nu_\psi.
\]
Since $S^{\Gamma}(A)$ is a closed face of $S_{\mu}(A)$, and $S_{\mu}(A)$ is a Choquet simplex, $S^{\Gamma}(A)$ is a Choquet simplex. Let \mbox{$\mu_\rho\in\prob(S^{\Gamma}(A))$} be the unique Choquet maximal representing measure for
$\rho$ on $S^{\Gamma}(A)$. Since
\mbox{$\partial_eS^{\Gamma}(A)=\partial_eS_\mu(A)\cap S^{\Gamma}(A)$}, $\mu_\rho$ is a Choquet maximal representing measure for $\rho$ when considered as a measure on $S_{\mu}(A)$. By uniqueness, it equals $\nu_\rho$. Hence
\[
    \nu_\rho(\partial_eS^{\Gamma}(A))=1,
\]
and therefore
\[
    \nu_\varphi(\partial_eS^{\Gamma}(A))\geq t>0.
\]
Conversely, assume that $\nu_\varphi(\partial_eS^{\Gamma}(A))>0$, and define
\[
    \tau_0(\,\cdot\,)
    =\int_{\partial_eS^{\Gamma}(A)}\omega(\,\cdot\,)\,d\nu_\varphi(\omega).
\]
Then $\tau_0$ is a nonzero positive invariant functional and
\mbox{$0\leq\tau_0\leq\varphi$}.
\end{proof}

If $S_{\mu}(A)$ is a Choquet simplex, Proposition \ref{prop:closed.face} implies $S^{\Gamma}(A)$ is a split face.
\cite[Theorem~9]{goodearl1976choquet}. Lemma~\ref{prop:purely.nontracial.iff.measure.0}
identifies its complementary face with $S^{\Gamma}(A)^\perp$. We record the resulting
decomposition explicitly.

\begin{thm}[Invariant--purely noninvariant decomposition]
\label{thm:trace-decomposition}
Assume $S_{\mu}(A)$ is a Choquet simplex. Let \mbox{$\varphi\in S_\mu(A)$} be neither invariant nor purely noninvariant. Then there are unique
$\rho\in  S^{\Gamma}(A)$, \mbox{$\psi\in S^{\Gamma}(A)^\perp$}, and $t\in(0,1)$ such that
\[
    \varphi=t\rho+(1-t)\psi.
\]
\label{thm:tracial.non.tracial.decomp}
\end{thm}

\begin{proof}
Let \mbox{$E=\partial_eS^{\Gamma}(A)$} and set
\[
    t=\nu_\varphi(E).
\]
By Lemma~\ref{prop:purely.nontracial.iff.measure.0}, we have $t>0$.
Moreover, $t<1$, since otherwise $\nu_\varphi$ would be supported on $E$, and hence
$\varphi$ would be invariant. Define
\[
    \rho=\frac{1}{t}\int_E
    \omega(\,\cdot\,)\,d\nu_\varphi(\omega),
    \qquad
    \psi=\frac{1}{1-t}
    \int_{\partial_eS_\mu(A)\setminus E}
    \omega(\,\cdot\,)\,d\nu_\varphi(\omega).
\]
Then $\rho\in S^{\Gamma}(A)$, \mbox{$\psi\in S_\mu(A)$}, and
\[
    \varphi=t\rho+(1-t)\psi.
\]
The normalized restrictions
\[
    \frac{1}{t}\left.\nu_\varphi\right|_E
    \quad\text{and}\quad
    \frac{1}{1-t}
    \left.\nu_\varphi\right|_{\partial_eS_\mu(A)\setminus E}
\]
are Choquet maximal representing measures for $\rho$ and $\psi$, respectively. By
uniqueness, they equal $\nu_\rho$ and $\nu_\psi$. In particular,
\[
    \nu_\psi(E)=0,
\]
so Lemma~\ref{prop:purely.nontracial.iff.measure.0} gives
\mbox{$\psi\in S^{\Gamma}(A)^\perp$}.
To prove uniqueness, suppose that
\[
    \varphi=s\rho'+(1-s)\psi',
\]
where $s\in(0,1)$, $\rho'\in S^{\Gamma}(A)$, and \mbox{$\psi'\in S^{\Gamma}(A)^\perp$}. Since
$s\nu_{\rho'}+(1-s)\nu_{\psi'}$ is a Choquet maximal representing measure for
$\varphi$, uniqueness gives
\[
    \nu_\varphi=s\nu_{\rho'}+(1-s)\nu_{\psi'}.
\]
We have \mbox{$\nu_{\rho'}(E)=1$}, while
Lemma~\ref{prop:purely.nontracial.iff.measure.0} gives \mbox{$\nu_{\psi'}(E)=0$}.
Evaluating the preceding equality on $E$ yields $s=t$. Restricting it to $E$ and to its
complement in $\partial_eS_\mu(A)$ then gives
\[
    \nu_{\rho'}=\nu_\rho
    \qquad\text{and}\qquad
    \nu_{\psi'}=\nu_\psi.
\]
Taking barycenters, we obtain $\rho'=\rho$ and $\psi'=\psi$.
\end{proof}

For the rest of this section, we assume $\Gamma\leq U(A)$ generates $A$ as a $C^*$-algebra, and the action is inner.
For extremal invariant states, a related trace/type~$\mathrm{III}$ dichotomy was
established by St{\o}rmer \cite{stormer1967types}. His theorem assumes that the cyclic
separating vector is invariant under every implementing unitary. Here the state is only
$\mu$-stationary, and no factoriality assumption is made.

\begin{thm}\label{thm:finite-typeIII}
Let
\mbox{$\varphi\in S_\mu(A)$}. Then,
\begin{align*}
    \varphi\in T(A)
    &\quad\Longleftrightarrow\quad M_\varphi\text{ is finite},\\
    \varphi\text{ is purely nontracial}
    &\quad\Longleftrightarrow\quad M_\varphi\text{ is of type }\mathrm{III}.
\end{align*}
\end{thm}

\begin{proof}
    Retain the notation $E_N$ and $E$.
    For the first part, assume $\varphi$ is a trace. The normal extension $\omega_{\varphi}$ is a faithful normal state on $M_{\varphi}$ by Lemma~\ref{lem:faithful.normal.state.on.vN}, and it is tracial
    since $\pi_{\varphi}(A)$ is ultraweakly dense in $M_{\varphi}$.
    In particular,
    $M_{\varphi}$ is tracial and thus finite.
    For the other direction, assume $M_{\varphi}$ is finite, and let
    \mbox{$\operatorname{Tr}_Z\colon M_{\varphi}\to Z(M_{\varphi})$} be its center-valued trace.
    Since $\operatorname{Tr}_Z$ is invariant under the conjugation action,
    \mbox{$\operatorname{Tr}_Z\circ T_{\varphi}=\operatorname{Tr}_Z$}. In particular,
    \mbox{$\operatorname{Tr}_Z\circ E_N=\operatorname{Tr}_Z$} for every \mbox{$N\in\mathbb{N}$},
    and normality gives \mbox{$\operatorname{Tr}_Z\circ E=\operatorname{Tr}_Z$}. Since
    \mbox{$\operatorname{Tr}_Z|_{Z(M_{\varphi})}=\mathrm{id}$}, for every \mbox{$x\in M_{\varphi}$},
    \[
    E(x)=\operatorname{Tr}_Z(E(x))=\operatorname{Tr}_Z(x).
    \]
    In particular, \mbox{$\operatorname{Tr}_Z=E$}, and thus by stationarity,
    \mbox{$\omega_{\varphi}\circ \operatorname{Tr}_Z=\omega_{\varphi}$}. Now given \mbox{$x,y\in M_{\varphi}$},
    \[
    \omega_{\varphi}(xy)=\omega_{\varphi}(\operatorname{Tr}_Z(xy))
    =\omega_{\varphi}(\operatorname{Tr}_Z(yx))=\omega_{\varphi}(yx),
    \]
    and $\omega_{\varphi}$ is a faithful, normal tracial state on $M_{\varphi}$. Restricting back to $\pi_{\varphi}(A)$, we have \mbox{$\varphi\in T(A)$}.

    \noindent For the second part, assume $M_{\varphi}$ is of type $\mathrm{III}$, and let \mbox{$0\leq \tau\leq \varphi$} be a positive tracial functional on $A$. By the GNS Radon--Nikodym theorem, there exists a positive contraction \mbox{$h\in\pi_{\varphi}(A)'$} such that, letting \mbox{$\zeta=h^{\frac{1}{2}}\xi_{\varphi}$},
    $$\tau(a)=\langle \pi_{\varphi}(a)\zeta,\zeta\rangle.$$ 
    Consider the normal extension of $\tau$ to $M_{\varphi}$ given by
    $$\omega_{\tau}(x)=\langle x\zeta,\zeta\rangle.$$
    Since $\tau$ is tracial on $A$ and $\omega_{\tau}$ is normal, ultraweak density shows
    that $\omega_{\tau}$ is a normal trace on $M_{\varphi}$. Since type $\mathrm{III}$ von
    Neumann algebras do not admit nonzero normal finite traces,
    \mbox{$\omega_{\tau}\equiv 0$}, and therefore $\tau\equiv 0$.
    Conversely, assume that \mbox{$\varphi\in T(A)^{\perp}$}.
By the type decomposition theorem, there exists a central projection \mbox{$z_{\mathrm{sf}}\in Z(M_{\varphi})$} 
such that
\[
M_{\varphi}
=
z_{\mathrm{sf}}M_{\varphi}\oplus (1-z_{\mathrm{sf}})M_{\varphi},
\]
where $z_{\mathrm{sf}}M_{\varphi}$ is semifinite and $(1-z_{\mathrm{sf}})M_{\varphi}$ is of type~$\mathrm{III}$.
Assume \mbox{$z_{\mathrm{sf}}\neq 0$}.
Since $z_{\mathrm{sf}}M_{\varphi}$ is semifinite, there exists a faithful
normal semifinite trace
\[
\mathrm{Tr}\colon (z_{\mathrm{sf}}M_{\varphi})_+\to [0,\infty].
\]
Let \mbox{$0\neq p\in (z_{\mathrm{sf}}M_{\varphi})_+$} be such that $\mathrm{Tr}(p)<\infty$.
Observe that, since \mbox{$z_{\mathrm{sf}}\in Z(M_{\varphi})$} and \mbox{$p=z_{\mathrm{sf}}p$},
$$T_{\varphi}(p)=T_{\varphi}(z_{\mathrm{sf}}p)=z_{\mathrm{sf}}T_{\varphi}(p),$$
and therefore \mbox{$E(p)\in z_{\mathrm{sf}}M_{\varphi}\cap Z(M_{\varphi})$}, which equals
\mbox{$Z(z_{\mathrm{sf}}M_{\varphi})$}.
Since $\mathrm{Tr}$ is conjugation-invariant, \mbox{$\mathrm{Tr}\circ T_{\varphi}=\mathrm{Tr}$}, and therefore
\mbox{$\mathrm{Tr}(E_N(p))=\mathrm{Tr}(p)$} for every \mbox{$N\in\mathbb{N}$}. By the lower
semicontinuity of $\mathrm{Tr}$ along the subnet defining $E$, we see that
\[
\mathrm{Tr}(E(p))\leq \liminf_\lambda\mathrm{Tr}(E_{N_\lambda}(p))
=\mathrm{Tr}(p)<\infty.
\]
Moreover, \mbox{$E(p)\neq 0$} since $$0<\omega_{\varphi}(p)=\omega_{\varphi}(E(p)).$$
Thus, we can replace $p$ with $E(p)$ and assume \mbox{$p\in Z(M_{\varphi})$}. Lastly,
take $\varepsilon>0$ such that
\mbox{$\widetilde p=1_{[\varepsilon,\infty)}(p)\neq 0$}. Since
\mbox{$\mathrm{Tr}(\widetilde p)\leq \varepsilon^{-1}\mathrm{Tr}(p)<\infty$}, after replacing
$p$ with $\widetilde p$ we may assume that $p$ is a nonzero central projection in
$M_{\varphi}$ with finite trace. In particular, $pM_{\varphi}$ is a finite von Neumann
algebra that admits the $\mu$-stationary state
$$\omega_{\mathrm{sf}}(x)=\frac{\omega_{\varphi}(px)}{\omega_{\varphi}(p)}.$$
By the same argument as in the first part, applied to $pM_{\varphi}$, this implies that
$\omega_{\mathrm{sf}}$ is a faithful normal tracial state on $pM_{\varphi}$. Define
\mbox{$\tau(a)=\omega_{\varphi}(p\pi_{\varphi}(a))$} for $a\in A$. Then $\tau$ is a positive tracial functional on $A$, and for every $a\in A_+$,
$$\varphi(a)-\tau(a)=\omega_{\varphi}(\pi_{\varphi}(a))-\omega_{\varphi}(p\pi_{\varphi}(a))=\omega_{\varphi}((1-p)\pi_{\varphi}(a))\geq 0.$$
Thus \mbox{$0\leq\tau\leq\varphi$}. Since \mbox{$\varphi\in T(A)^\perp$}, we have $\tau=0$, which
contradicts \mbox{$\tau(1)=\omega_{\varphi}(p)>0$}. Therefore \mbox{$z_{\mathrm{sf}}=0$}, and $M_{\varphi}$ is
of type~$\mathrm{III}$.
\end{proof}
\begin{cor}
Let \mbox{$\varphi\in\partial_eS_\mu(A)$}. If $\varphi$ is tracial, then $M_\varphi$ is a finite
factor. If $\varphi$ is not tracial, then $M_\varphi$ is a type~$\mathrm{III}$ factor.
\end{cor}
\begin{cor}
    Let \mbox{$\varphi\in S_{\mu}(A)\setminus\left(T(A)\cup T(A)^{\perp}\right)$}, and write 
    $$\varphi=t\tau+(1-t)\psi,\quad \tau\in T(A),\,\psi\in T(A)^{\perp},\quad t\in (0,1)$$
    as in Theorem \ref{thm:tracial.non.tracial.decomp}. Then \mbox{$M_{\varphi}\cong M_{\tau}\oplus M_{\psi}$}. In particular, $M_{\varphi}$ has no semifinite, non-finite part in its type decomposition.
\end{cor}

\begin{proof}
   First, recall that two representations $\rho,\,\pi$ of $A$ are called \textit{disjoint} if no nontrivial subrepresentation of $\pi$ is unitarily equivalent to a subrepresentation of $\rho$ \cite[Section 5.2]{dixmier1977cstar}. By Theorem \ref{thm:finite-typeIII},  $M_{\tau}$ is finite, and  $M_{\psi}$ is type $\mathrm{III}$. By \cite[5.6.3]{dixmier1977cstar}, $\pi_{\tau}$ and $\pi_{\psi}$ are disjoint. Now let \mbox{$\Pi=\pi_{\tau}\oplus\pi_{\psi}$} and $a_0\in A$. If $\begin{pmatrix}
       x&y
       \\
       z&w
   \end{pmatrix}\in \Pi(A)'$, then
   $$\begin{pmatrix}
       \pi_{\tau}(a)x&\pi_{\tau}(a)y\\
       \pi_{\psi}(a)z&\pi_{\psi}(a)w
   \end{pmatrix}=\begin{pmatrix}
       \pi_{\tau}(a)&0
       \\
       0&\pi_{\psi}(a)
   \end{pmatrix}\begin{pmatrix}
       x&y
       \\
       z&w
   \end{pmatrix}=\begin{pmatrix}
       x&y
       \\
       z&w
   \end{pmatrix}\begin{pmatrix}
       \pi_{\tau}(a)&0
       \\
       0&\pi_{\psi}(a)
   \end{pmatrix}=\begin{pmatrix}
       x\pi_{\tau}(a)&y\pi_{\psi}(a)
       \\
       z\pi_{\tau}(a)&w\pi_{\psi}(a)
   \end{pmatrix}.$$
   In particular, \mbox{$x\in \pi_{\tau}(A)'$}, \mbox{$w\in\pi_{\psi}(A)'$}, and 
   $y,z$ are intertwiners of $\pi_{\tau}$ and $\pi_{\psi}$. Since these representations are disjoint, $y=z=0$, and
   $$\Pi(A)'=\pi_{\tau}(A)'\oplus\pi_{\psi}(A)'.$$
   The same calculation gives $$\Pi(A)''=M_{\tau}\oplus M_{\psi}.$$ Lastly, by Lemma \ref{lem:quasi.equiv.of.GNS.for.convex.comb}, 
   $$\Pi(A)''\cong M_{\varphi},$$
   and the conclusion follows.
\end{proof}
\section{Examples and applications}

We finish with a few results concerning the geometry of the stationary state simplex under an inner action of a generating subgroup $\Gamma\leq U(A)$ on $A$.

\begin{cor}
Let \mbox{$2\leq d\leq\infty$}, let \mbox{$\mu\in\prob(F_d)$} be a generating measure, and let $F_d$
act on $C^*(F_d)$ by conjugation. Then $S_\mu(C^*(F_d))$ has a Poulsen face.
\end{cor}

\begin{proof}
By \cite[Theorem~1.1]{orovitz2024free}, the trace simplex $T(C^*(F_d))$ is the Poulsen
simplex. 
\end{proof}

For comparison, Glasner and Weiss proved that for a property $(T)$ group $\Gamma$ acting on a compact Hausdorff space $X$, $\prob^\Gamma(X)$ is always Bauer \cite{glasner1997kazhdan}. Moreover, Levit, Slutsky, and Vigdorovich proved that the trace simplex of a
property~$(T)$ group is Bauer and that every finite-dimensional character is isolated
in its extreme boundary \cite[Corollaries~1.6 and~1.8]{levit2023spectral}. The next
result shows that the trivial character remains isolated in the extreme boundary of the
larger stationary state simplex.

\begin{cor}
Let $\Gamma$ be a nontrivial countable discrete group with property~$(T)$.
Let \mbox{$\mu\in\prob(\Gamma)$} be a generating measure. Suppose $\Gamma$ acts on
$C^*(\Gamma)$ by conjugation. Then $S_\mu(C^*(\Gamma))$ is not a Poulsen simplex.
\end{cor}

\begin{proof}
Let \mbox{$p_\Gamma\in C^*(\Gamma)$} be the Kazhdan projection. Thus $p_\Gamma$ is a central
projection such that, for every unitary representation $\pi$ of $\Gamma$,
$\pi(p_\Gamma)$ is the orthogonal projection onto the invariant vectors of $\pi$
\cite{valette1984minimal}.
Let \mbox{$\varphi\in\partial_eS_\mu(C^*(\Gamma))$}. By
Theorem~\ref{thm:extreme-factor}, $M_\varphi$ is a factor. Since
$\pi_\varphi(p_\Gamma)$ is central in $M_\varphi$, it is either $0$ or $1$. If
\mbox{$\pi_\varphi(p_\Gamma)=1$}, then every vector in $H_\varphi$ is invariant under
$\pi_\varphi(\Gamma)$, and hence \mbox{$\varphi=1_\Gamma$}. Consequently,
\[
 \varphi(p_\Gamma)=
 \begin{cases}
  1,&\varphi=1_\Gamma,\\
  0,&\varphi\in\partial_eS_\mu(C^*(\Gamma))\setminus\{1_\Gamma\}.
 \end{cases}
\]
It follows that $1_\Gamma$ is isolated in the extreme boundary of
$S_\mu(C^*(\Gamma))$. Since a nontrivial Poulsen simplex has no isolated extreme points
\cite{lindenstrauss1978poulsen}, $S_\mu(C^*(\Gamma))$ is not Poulsen.
\end{proof}
Let $\varphi\in S_{\mu}(A)$, and let $(\pi_{\varphi}, H_{\varphi},\xi_{\varphi})$ be its GNS representation. By Lemma \ref{lem:faithful.normal.state.on.vN}, $\varphi$ extends to a normal faithful $\mu$-stationary state on $M_{\varphi}$. We define a norm on $M_{\varphi}$ by
$$\|x\|^2_{\omega_{\varphi}}=\omega_{\varphi}(x^*x)=\|x\xi_{\varphi}\|_{H_{\varphi}}^2,\quad x\in M_{\varphi}.$$
We also define the seminorm on $A$
$$\|a\|^2_{\varphi}=\varphi(a^*a)=\|\pi_{\varphi}(a)\|^2_{\omega_{\varphi}}.$$
\begin{prop}
    Assume there exists a finite set $F\subset\mathbb{N}$ and $\varepsilon>0$ such that for every $\varphi\in\partial_eS_{\mu}(A)$ and every $a\in A_{sa}$ with $\|a\|_{\varphi}=1$ and $\varphi(a)=0$,
    $$\max_{m\in F}\|\Phi^{(m)}_{\mu}(a)-a\|_{\varphi}\geq \varepsilon.$$
    Then $S_{\mu}(A)$ is a Bauer simplex.
    \label{prop:bauer.simplex.sufficient.condition}
\end{prop}
\begin{proof}
    Assume the contrary, and let $\varphi_n\in\partial_eS_{\mu}(A)$ such that 
    $$\varphi_n\xrightarrow[n\to\infty]{w-*}\varphi,\quad \varphi\not\in\partial_eS_{\mu}(A).$$ 
    By Theorem \ref{thm:extreme-factor}, $Z(M_{\varphi})\neq \mathbb{C}1$. In particular, there exists a self-adjoint element $z\in Z(M_{\varphi})$ that is non-scalar. Up to replacing $z$ with $\displaystyle\frac{z-\omega_{\varphi}(z)1}{\|z-\omega_{\varphi}(z)1\|_{\omega_{\varphi}}}$, we may assume 
    $$\omega_{\varphi}(z)=0,\,\|z\|_{\omega_{\varphi}}=1.$$
    Fix $0<\delta<\min\left\{\frac{\varepsilon}{6},\,\frac{1}{8}\right\}$.
    By the self-adjoint version of Kaplansky's theorem, we may take $a\in A_{sa}$ such that
    $$\|\pi_{\varphi}(a)\xi_{\varphi}-z\xi_{\varphi}\|_{H_{\varphi}}=\|\pi_{\varphi}(a)-z\|_{\omega_{\varphi}}<\delta.$$ Therefore,
    $$|\varphi(a)|=|\omega_{\varphi}(\pi_{\varphi}(a)-z)|\leq \|\pi_{\varphi}(a)-z\|_{\omega_{\varphi}}<\delta.$$
    Moreover, by the reverse triangle inequality, and since $\|z\|_{\omega_{\varphi}}=1$,
    $$1-\delta<\|a\|_{\varphi}<1+\delta.$$
    By Lemma \ref{lem:fixed-points}, $T_{\varphi}(z)=z$. Therefore, for every $m\in F$,
    \begin{align*}
\bigl\|\Phi_{\mu}^{(m)}(a)-a\bigr\|_{\varphi}
&=
\bigl\|
\pi_{\varphi}\bigl(\Phi_{\mu}^{(m)}(a)-a\bigr)
\bigr\|_{\omega_{\varphi}}\\
&=
\bigl\|
\pi_{\varphi}\bigl(\Phi_{\mu}^{(m)}(a)\bigr)
-\pi_{\varphi}(a)
\bigr\|_{\omega_{\varphi}}
\\
&=
\bigl\|
T_{\varphi}^{(m)}\bigl(\pi_{\varphi}(a)\bigr)
-\pi_{\varphi}(a)
\bigr\|_{\omega_{\varphi}}
\\
&\leq
\bigl\|
T_{\varphi}^{(m)}\bigl(\pi_{\varphi}(a)\bigr)
-T_{\varphi}^{(m)}(z)
\bigr\|_{\omega_{\varphi}}
+
\bigl\|z-\pi_{\varphi}(a)\bigr\|_{\omega_{\varphi}}
\\
&\leq
\bigl\|\pi_{\varphi}(a)-z\bigr\|_{\omega_{\varphi}}
+
\bigl\|z-\pi_{\varphi}(a)\bigr\|_{\omega_{\varphi}}
\\
&=
2\bigl\|z-\pi_{\varphi}(a)\bigr\|_{\omega_{\varphi}}
\\
&<2\delta.
\end{align*}
We used the fact that $T_{\varphi}$ is a contraction on $M_{\varphi}$ with respect to $\|\cdot\|_{\omega_{\varphi}}$. Since $F$ is finite and \mbox{$\varphi_n\xrightarrow[n\to\infty]{w-*}\varphi$}, for all sufficiently large $n$ we have
\begin{gather*}
\lvert\varphi_n(a)\rvert<2\delta,\\
1-2\delta
<\lVert a\rVert_{\varphi_n}
<1+2\delta,\\
\max_{m\in F}
\bigl\lVert\Phi_\mu^{(m)}(a)-a\bigr\rVert_{\varphi_n}
<3\delta.
\end{gather*}
Fix such an $n\in\mathbb{N}$, and replace $a$ with $b_n=a-\varphi_n(a)1\in A_{sa}$. Then
\begin{gather*}
\varphi_n(b_n)=0\quad \text{and}\quad
\max_{m\in F}
\bigl\lVert\Phi_\mu^{(m)}(b_n)-b_n\bigr\rVert_{\varphi_n}
<3\delta.
\end{gather*}
Moreover, 
$$\|b_n\|^2_{\varphi_n}=\varphi_n((a-\varphi_n(a)1)^2)=\|a\|_{\varphi_n}^2-|\varphi_n(a)|^2>(1-2\delta)^2-4\delta^2=1-4\delta>\frac{1}{2}.$$
In particular, $\|b_n\|_{\varphi_n}>\frac{1}{\sqrt{2}}$. Let $c_n=\frac{b_n}{\|b_n\|_{\varphi_n}}$. Then 
$$\varepsilon\leq \max_{m\in F}\|\Phi_{\mu}^{(m)}(c_n)-c_n\|_{\varphi_n}=\frac{1}{\|b_n\|_{\varphi_n}}\max_{m\in F}\|\Phi_{\mu}^{(m)}(b_n)-b_n\|_{\varphi_n}\leq 3\sqrt{2}\delta<6\delta<\varepsilon,$$
which is a contradiction.
\end{proof}
We now prove a theorem similar to \cite[Theorem 3.1]{slutsky2026invariant}.
\begin{thm}
    Assume there exists $\varepsilon>0$ and a finite set $K\subset \mathrm{supp}(\check\mu)$ such that
    $$\max_{g\in K}\|\alpha_g(a)-a\|_{\varphi}\geq\varepsilon$$
    for every $\varphi\in \partial_eS_{\mu}(A)$ and every $a\in A_{sa}$ such that $\varphi(a)=0$ and $\|a\|_{\varphi}=1$. Then $S_{\mu}(A)$ is a Bauer simplex.
    \label{thm:bauer.simplex}
\end{thm}
\begin{proof}
    First of all, we observe that 
      $$\|\Phi^{(m)}_{\mu}(a)-a\|_{\varphi}\leq \sum_{k=0}^{m-1}\|\Phi^{(k+1)}_{\mu}(a)-\Phi_{\mu}^{(k)}(a)\|_{\varphi}\leq m\|\Phi_{\mu}(a)-a\|_{\varphi}.$$
      Therefore, it is always enough to check the condition of Proposition \ref{prop:bauer.simplex.sufficient.condition} on $\|\Phi_{\mu}(a)-a\|_{\varphi}$ (equivalently, take the finite set $F$ in Proposition \ref{prop:bauer.simplex.sufficient.condition} to be \mbox{$F=\{1\}$)}. Moreover, 
    the condition in Proposition \ref{prop:bauer.simplex.sufficient.condition} is equivalent to the existence of $\varepsilon'>0$ such that for every $\varphi\in\partial_eS_{\mu}(A)$ and every $a\in A_{sa}$ satisfying $\|a\|_{\varphi}=1$ and $\varphi(a)=0$,
    $$\sum_{g\in\Gamma}\check\mu(g)||\alpha_g(a)-a\|^2_{\varphi}\geq\varepsilon'.$$
    Indeed, for every $\varphi\in S_{\mu}(A)$ and $a\in A_{sa}$ that satisfies the conditions,
    \begin{align*}
    \sum_{g\in\Gamma}\check\mu(g)\|\alpha_g(a)-a\|^2_{\varphi}=2\|a\|_{\varphi}^2-2\mathrm{Re}(\varphi(a\Phi_{\mu}(a)))\\=-2\mathrm{Re}(\varphi(a(\Phi_{\mu}(a)-a)))\\\leq 2\|\Phi_{\mu}(a)-a\|_{\varphi},
    \end{align*}
    where the last inequality follows from Cauchy--Schwarz and $\|a\|_{\varphi}=1$.
    Therefore, using Jensen's inequality and the triangle inequality for seminorms, we get the chain of inequalities
    \begin{equation}
    \frac{1}{2}\left(\sum_{g\in\Gamma}\check\mu(g)\|\alpha_g(a)-a\|_{\varphi}\right)^2\leq\frac{1}{2}\sum_{g\in\Gamma}\check\mu(g)\|\alpha_g(a)-a\|_{\varphi}^2\le \|\Phi_{\mu}(a)-a\|_{\varphi}.\label{eq:chain.of.inequalities}
    \end{equation}
    In particular, 
    $$\frac{1}{2}\varepsilon^2\min_{g\in K}\check\mu(g)^2\leq \frac{1}{2}\left(\sum_{g\in\Gamma}\check\mu(g)\|\alpha_g(a)-a\|_{\varphi}\right)^2\leq \|\Phi_{\mu}(a)-a\|_{\varphi}.$$
    The conclusion now follows.
\end{proof}
\begin{remark}
    In fact, equation \eqref{eq:chain.of.inequalities} combined with the inequality 
    $$\|\Phi_{\mu}(a)-a\|_{\varphi}\leq \sum_{g\in\Gamma}\check\mu(g)\|\alpha_g(a)-a\|_{\varphi}$$ shows that verifying the condition in Proposition \ref{prop:bauer.simplex.sufficient.condition} is equivalent to verifying the condition in Theorem \ref{thm:bauer.simplex}.
\end{remark}
The following example is a family of $(A_{d},\Gamma_{d},\mu_d)$ such that $S_{\mu_d}(A_d)$ is a Bauer simplex. Before introducing that,
we present some definitions. Let \mbox{$A$} be a unital \mbox{$C^*$}-algebra with a state \mbox{$\rho$}. A family of unital subalgebras \mbox{$(A_i)_{i\in I}$} is \emph{free with respect to $\rho$} if \mbox{$\rho(a_1\cdots a_n)=0$} whenever \mbox{$a_k\in A_{i_k}$}, \mbox{$\rho(a_k)=0$}, and consecutive indices differ \cite[Definition~1.3]{V}. For unital \mbox{$C^*$}-algebras \mbox{$A_i$} with states \mbox{$\rho_i$} having faithful GNS representations, we denote their reduced free product by \mbox{$(A,\rho)=*_{i\in I}(A_i,\rho_i)$}; its construction and canonical embeddings are recalled in \cite[Section~1, pp.~2--3]{RX}. The state-preserving conditional expectations onto its factors are given in \cite[Lemma~1.1]{BD}, with amalgamation over \mbox{$\mathbb C1$}. For von Neumann algebras equipped with faithful normal states, the corresponding free product is characterized by freely independent, state-preserving copies of the factors generating the ambient von Neumann algebra; see \cite[Section~1.3, Proposition~1.4]{V}.
\begin{example}
    Fix $d\geq 1$, let $S\subset\mathcal{O}_2$ be a finite symmetric generating set of unitaries, and let $\mathbb{F}_{\infty}$ denote the free group with infinite number of generators. Let $\nu\in\prob(S)$ be the uniform measure, and let $\psi\in S_{\nu}(\mathcal{O}_2)$. Moreover, let $\tau$ be the unique trace on $C^*_r(\mathbb{F}_{\infty})$, and define
    $$(C,\omega)=(\mathcal{O}_2,\psi)\ast (C^*_r(\mathbb{F}_{\infty}),\tau).$$
    Here $\ast$ denotes the free reduced product. In particular, $\omega|_{\mathcal{O}_2}=\psi$, $\omega
    |_{C^*_r(\mathbb{F}_\infty)}=\tau$, and $\mathcal{O}_2,\,C^*_r(\mathbb{F_{\infty}})$ are free with respect to $\omega$. Lastly, define 
    $$A_d=C(\mathbb{T}^d)\otimes_{\min}C.$$
    Let $\{r_j\}_{j=1}^{\infty}$ be the free unitaries, $\{z_j\}_{j=1}^d$ the coordinate unitaries in $C(\mathbb{T}^d)$, and let, up to suppressing the tensor embeddings,
    $$\Gamma_d=\langle S,z_1,\ldots,z_d,r_1,r_2,\ldots\rangle\leq U(A_d),\quad \mu_d=\frac{1}{3}\left(\nu+\frac{1}{2d}\sum_{i=1}^d (\delta_{z_i}+\delta_{z_i^{-1}})+\frac{1}{2}\sum_{j\geq 1}\frac{\delta_{r_j}+\delta_{r_j^{-1}}}{j(j+1)}\right),$$
    and assume $\Gamma_d$ acts on $A_d$ by conjugation.
    Then 
    $$\max_{j\in \{1,2\}}\|r_jar_j^{-1}-a\|_{\varphi}\geq\frac{1}{14}$$
    for every $\varphi\in \partial_eS_{\mu_d}(A_d)$ and every $a\in (A_d)_{\mathrm{sa}}$ satisfying $\|a\|_{\varphi}=1$ and $\varphi(a)=0$. By Theorem \ref{thm:bauer.simplex}, $S_{\mu_d}(A_d)$ is a Bauer simplex. We leave the technical details for the interested reader in Section \ref{sec:Appendix}.
    \label{ex:bauer.simplex.family}
\end{example}
\section{Appendix}\label{sec:Appendix}
In this section, we provide proofs for the technical parts of Example \ref{ex:bauer.simplex.family}. We will sometimes suppress the tensor embeddings, when no ambiguity arises. For $d\geq 1$, let $(A_d,\Gamma_d,\mu_d)$ be as in Example \ref{ex:bauer.simplex.family}. Observe that since $C(\mathbb{T}^d)\otimes 1$ is central in $A_d$,
    $$\Phi_{\mu_d}=\frac{1}{3}\left[I+\frac{1}{|S|}\sum_{s\in S}\mathrm{Ad}(s)+\frac{1}{2}
\sum_{\substack{j>\ell\\ \sigma\in\{\pm1\}}}
\frac{1}{j(j+1)}
\operatorname{Ad}(r_j^\sigma)
\right].$$
Let $$E_{\mathcal{O}_2}\colon C\to \mathcal{O}_2,\,E_{C^*_r(\mathbb{F}_{\infty})}\colon C\to C^*_r(\mathbb{F}_{\infty})$$
denote the conditional expectations. Since 
$$\omega=\psi\circ E_{\mathcal{O}_2}=\tau\circ E_{C^*_r(\mathbb{F}_{\infty})},$$
it is straight-forward to verify that $E\Phi_{\mu_d}=E$, where $E=\mathrm{id}\otimes\omega$, that $\omega\circ P_B=\omega$, where $P_B(x)=|S|^{-1}\sum_{s\in S}sxs^*$ for $x\in C$, and 
$$\omega(r_j^{\pm 1}x)=\omega(xr_j^{\pm 1}),\,j\in\mathbb{N},\,x\in C.$$
Now for $\ell\geq 1$, 
define $C_{\ell}=C^*(\mathcal{O}_2,r_1,\ldots,r_{\ell})$,
set
$t_{\ell}=\frac{1}{3(\ell+1)}$ and define the following operators on $A_d$
$$R_{\ell}=\sum_{\substack{j>\ell\\\sigma\in\{\pm 1\}}}
\frac{\ell+1}{2j(j+1)}
\operatorname{Ad}(r_j^\sigma),\quad Q_{\ell}=\frac{\Phi_{\mu_d}-t_{\ell}R_{\ell}}{1-t_{\ell}}.$$
Note that $Q_{\ell}$ is a contraction. In particular, $$\Phi_{\mu_d}=t_{\ell}R_{\ell}+(1-t_{\ell})Q_{\ell}$$
and
$Q_{\ell}(C_{\ell}\cap \ker(\omega))\subset C_{\ell}\cap \ker(\omega)$.
\begin{lem}
    For every $\ell\geq 1$ and every $x\in C_{\ell}\cap \ker(\omega)$, 
    $$\|R_{\ell}x\|\leq \underbrace{\frac{2}{\sqrt{2(\ell +1)}}}_{\eta_\ell}\|x\|.$$
\end{lem}
\begin{proof}
    Let $H_{\omega}$ be the GNS Hilbert space of $\omega$, and let $p\in B(H_{\omega})$ be the projection onto the closed subspace
    $$\overline{\mathrm{span}\{\pi_{\omega}(a)\xi_{\omega}:a\text{ is reduced and starts with a centered word in $\mathcal{O}_2$ or $r_j^{\pm 1},\,1\leq j\leq \ell$}\}}.$$
    Then $(1-p)x(1-p)=0$ for every $x\in C_{\ell}\cap\ker(\omega)$, and thus we can write every such $x$ as
    $$x=px+(1-p)xp.$$
    Fix $x\in C_{\ell}\cap\ker(\omega)$. For every $g=r_j^{\pm 1},\,j>\ell$,
    define 
    $$A_g=gpxg^{*},\,B_g=g(1-p)xpg^*.$$
    Then the $A_g$'s have orthogonal ranges. Similarly, $B_g^*$ have orthogonal ranges. Therefore, for every $n>\ell$,
    $$\left\|\sum_{n\geq j>\ell,\,\sigma\in\{\pm 1\}}\frac{\ell +1}{2j(j+1)}A_{r_j^{\sigma}}\right\|,\,\left\|\sum_{n\geq j>\ell,\,\sigma\in\{\pm 1\}}\frac{\ell +1}{2j(j+1)}B_{r_j^{\sigma}}\right\|\leq \|x\|\left(\sum_{n\geq j>\ell,\,\sigma\in\{\pm 1\}}\left(\frac{\ell +1}{2j(j+1)}\right)^2\right)^{\frac{1}{2}},$$
    where the inequality follows from the Pythagorean theorem. Therefore, passing to limits,
    $$\|R_{\ell}x\|\leq 2\|x\|\left(\sum_{ j>\ell,\,\sigma\in\{\pm 1\}}\frac{\ell +1}{2j(j+1)}\right)^{\frac{1}{2}}.$$
    Since
    $$\sum_{j>\ell}\frac{\ell+1}{2j(j+1)}=\frac{1}{2},\quad \max_{j>\ell}\frac{\ell+1}{2j(j+1)}=\frac{1}{2(\ell +2)},$$
    we have
    $$\|R_{\ell}x\|\leq \frac{2}{\sqrt{2(\ell +1)}}\|x\|.$$

\end{proof}
\begin{lem}
    Let $\varphi\in S_{\mu_d}(A_d)$ and identify $C$ with $1\otimes C$. Then $\varphi|_C=\omega$.
    \label{lem:restriction.of.stationary.equals.omega}
\end{lem}
\begin{proof}
    For $\ell\geq 1$, let 
    $$b_{\ell}=\sup\{|\varphi(x)|:x\in C_{\ell}\cap\ker(\omega),\,\|x\|\leq 1\}.$$
    For $x\in C_{\ell}\cap\ker(\omega)$ with $\|x\|\leq 1$, by stationarity,
    $$|\varphi(x)|=|\varphi(\Phi_{\mu_d}(x))|\leq (1-t_{\ell})|\varphi(Q_{\ell}(x))|+t_{\ell}|\varphi(R_{\ell}(x))|\leq (1-t_{\ell})b_{\ell}+t_{\ell}\eta_\ell,$$
    and in particular $b_{\ell}\leq\eta_{\ell}$. Now for fixed $m$ and $x\in C_{m}\cap \ker(\omega)$, applying this for every $\ell\geq m$ gives us $\varphi(x)=0$. Since $\displaystyle\overline{\bigcup_{m=1}^{\infty}C_m}=C$, we see that $\omega(x)=0$ implies $\varphi(x)=0$. In particular, for every $x\in C$,
    $$0=\varphi(x-\omega(x)1)=\varphi(x)-\omega(x),$$
    as desired.
\end{proof}
\begin{lem}
    Let $\varphi\in\partial_eS_{\mu}(A)$. Denote $N_1=W^*(\pi_{\varphi}(r_1)),\,N_2=W^*(\pi_{\varphi}(\mathcal{O}_2), \pi_{\varphi}(r_2),\pi_{\varphi}(r_3),\ldots)$. Then
    $$(M_{\varphi},\omega_\varphi)\cong (N_1,\omega_{\varphi}|_{N_1})\ast (N_2,\omega_{\varphi}|_{N_2}),$$
    with $\ast$ denoting the reduced free product. 
\end{lem}
\begin{proof}
    By Lemma \ref{lem:faithful.normal.state.on.vN}, $\omega_{\varphi}$ is faithful. Moreover, by Lemma \ref{lem:restriction.of.stationary.equals.omega}, generating $C^*$-subalgebras for $N_1$ and $N_2$ are free with respect to $\varphi$. By Kaplansky's density theorem, $N_1$ and $N_2$ are free with respect to $\varphi$. Moreover, they generate $M_{\varphi}$ since $M_{\varphi}$ is a factor and the image of $C(\mathbb{T}^d)$ under $\pi_{\varphi}$ is central (and thus scalar). 
    By \cite[Proposition 1.4]{V}, we get the desired isomorphism.
\end{proof}
\begin{lem}[{\cite[Lemma~4.1]{V}}]
Let $N_i$ be von Neumann algebras equipped with faithful normal
states $\omega_i$, for $i=1,2$, and put
\[
(N,\omega)=(N_1,\omega_1)*(N_2,\omega_2).
\]
Let $a\in N_1$ and $b,c\in N_2$ belong to the domain of
$\sigma^\omega_{i/2}$, where $(\sigma^\omega_t)_{t\in\mathbb R}$
is the modular automorphism group of $\omega$.
For $i=1,2$, let $\alpha_i$ be an automorphism of $N_i$
satisfying $\omega_i\circ\alpha_i=\omega_i$, and write
$\alpha=\alpha_1*\alpha_2$.
Then, for every $x\in N$,
\[
\|x-\omega(x)1\|_{\omega}
\le E(a,b,c)\max_{v\in\{a,b,c\}}\|xv-\alpha(v)x\|_{\omega}
   +F(a,b,c)\|x\|_{\omega},
\]
where $\|x\|_{\omega}=\omega(x^*x)^{1/2}$ and
\[
\begin{aligned}
E(a,b,c)
&=6\|a\|^3+4\|b\|^3+4\|c\|^3,\\
F(a,b,c)
&=3C(a)+2C(b)+2C(c)
  +12|\omega(cb^*)|\,\|cb^*\|,\\
C(v)
&=2\|v\|^3\|\sigma^\omega_{i/2}(v)-v\|
  +2\|v\|^2\|v^*v-1\|\\
&\quad+3(1+\|v\|^2)\|vv^*-1\|
  +6|\omega(v)|\,\|v\|.
\end{aligned}
\]
\end{lem}
\begin{lem}
For every $\varphi\in \partial_eS_{\mu_d}(A_d)$ and every $x\in (A_d)_{\mathrm{sa}}$ satisfying $\|a\|_{\varphi}=1$ and $\varphi(x)=0$,
    $$\max_{j\in \{1,2\}}\|r_jxr_j^{-1}-x\|_{\varphi}\geq\frac{1}{14}.$$
\end{lem}
\begin{proof}
      Fix $\varphi\in\partial_eS_{\mu_d}(A_d)$ and let $\alpha_i=Id_{N_i}$. This implies $\alpha=Id_{M_{\varphi}}$.
    Take $a=\pi_{\varphi}(r_1)$, \mbox{$b=\pi_{\varphi}(r_2),\,c=b^{*}$}. Since these are unitaries, $E(a,b,c)=14$. 
    Moreover, since 
    \mbox{$\omega(r_1y)=\omega(yr_1)$} for every $y\in C$, we have $\omega_{\varphi}(a\pi_{\varphi}(y))=\omega_{\varphi}(\pi_{\varphi}(y)a)$ for every $y\in C$. By normality, $\omega_{\varphi}(ax)=\omega_{\varphi}(xa)$ for every $x\in M_{\varphi}$, and in particular the modular automorphism group of $\omega_{\varphi}$ fixes $a$. By analytic continuation, $\sigma^{\omega_{\varphi}}_{\frac{i}{2}}(a)=a$. 
    Moreover, 
    $$\omega_{\varphi}(a)=\omega(r_1)=\tau(r_1)=0.$$
    This gives us $C(a)=0$, and similarly for $b$ and $c$. Lastly, $\omega_{\varphi}(cb^*)=\tau(r_2^{-2})=0$, and thus $F(a,b,c)=0$. We conclude
     $$\|x-\omega_{\varphi}(x)1\|_{\omega_{\varphi}}\leq 14\max\{\|xa-x\|_{\omega_{\varphi}},\,\|xb-x\|_{\omega_{\varphi}},\,\|xc-x\|_{\omega_{\varphi}}\}.$$
    Since right multiplication by elements in $\{a,b,c\}$ preserves the $\omega_{\varphi}$ norm (by what we showed),
    $$\|x-\omega_{\varphi}(x)1\|_{\omega_{\varphi}}\leq 14\max\{\|axa^*-x\|_{\omega_{\varphi}},\|bxb^*-x\|_{\omega_{\varphi}}\}.$$
    Replacing $x$ with $\pi_{\varphi}(y)$ for $y$ satisfying $\|y\|_{\varphi}=0$ and $\varphi(y)=1$, we see
    $$\max_{j=1,2}\|r_jyr_j^{-1}-y\|_{\varphi}\geq\frac{1}{14}.$$
\end{proof}
\subsection*{Acknowledgments.}
I thank Eli Shamovich and Matthew Kennedy for their interesting talks, and Eli Shamovich for his helpful comments on earlier drafts of this paper.

During the preparation of this manuscript, the author used generative AI tools to assist with drafting portions of the Introduction and Preliminaries, improving the presentation and typesetting of the text and equations, proofreading for typographical errors and inconsistencies in notation, organizing the bibliography, and locating relevant references. Moreover, it provided great help in finding Example \ref{ex:bauer.simplex.family}. These tools were not used to produce the other mathematical results, arguments, or proofs, which are entirely the author’s own. All cited sources were independently verified by the author, who assumes full responsibility for the final content of the manuscript.

\bibliographystyle{alpha}
\bibliography{my_bib}
\end{document}